\documentclass{birkmult}

\usepackage[latin1]{inputenc}
\usepackage{subfigure}
\usepackage{graphicx}
\usepackage{amsmath, amsfonts}
\newtheorem{theorem}{Theorem}[section]
\newtheorem{proposition}{Proposition}[section]
\newtheorem{lemma}[theorem]{Lemma}
\newtheorem{definition}[theorem]{Definition}
\newtheorem{example}[theorem]{Example}

\newtheorem{remark}[theorem]{Remark}

\def\qed{\hspace{3.5mm} \hfill \vbox{\hrule height 3pt depth 2 pt width 2mm}}

\def\binomial#1#2{
\left(
\begin{matrix}
#1 \\ #2
\end{matrix}
\right)}

\def\U0{U_0({ G}l_N)}

\def\<{\langle}
\def\>{\rangle}

\begin{document}

\title{Bijective Enumeration of 3-Factorizations of an N-Cycle}
\author{Ekaterina Vassilieva}

\address{Laboratoire d'Informatique de l'Ecole Polytechnique, 91128 Palaiseau, FRANCE}

\email{ekaterina.vassilieva@lix.polytechnique.fr}




\date{\today}


\begin{abstract}
This paper is dedicated to the factorizations of the symmetric group. Introducing a new bijection for partitioned 3-cacti, we derive an elegant formula for the number of factorizations of a long cycle into a product of three permutations. As the most salient aspect, our construction provides the first purely combinatorial computation of this number.  
\end{abstract}

\maketitle

\section{Introduction}

Let $S_N$ be the symmetric group on $N$ symbols and $\gamma_N=(1 \, 2 \, \ldots \, N)$ be the long cycle in $S_N$. This paper is focused on the enumeration of the number of factorizations of $\gamma_N$  as a product of three permutations  $\alpha_1$, $\alpha_2$, $\alpha_3$ with a given number of cycles $n_1$, $n_2$, $n_3$. Factorizations in the symmetric group received significant attention over the past two decades. Among the most celebrated results in the field, Harer and Zagier in \cite{HZ} enumerated the factorization of the long cycle as a product of a permutations and a fixed point free involution. Later on, Jackson (\cite{J}) derived a more general formula enumerating the factorizations of a permutation into an arbitrary number of permutations with a given number of cycles. Both works rely on character theoretical techniques and leave little room for combinatorial interpretation. The first combinatorial proof of the Harer-Zagier formula was given recently by Lass in \cite{L}. In \cite{GN}, Goulden and Nica  developed another combinatorial proof relying on a direct bijection. The first combinatorial approach to the case involving two general permutations was developed by Goupil and Schaeffer in \cite{GS}. Subsequently, Schaeffer and Vassilieva used a bijective method in \cite{SV} to address the same issue. 
This later article relies on the direct correspondence between $2$-factorizations of a permutation and graphs with black and white vertices embedded in surfaces called {\it bicolored maps}. Our purpose is to start with a similar equivalence to solve the $3$-factor case by means of a combinatorial development. While at first glance the $2$-factor and the  $3$-factor cases seem to share common characteristics, a number of difficulties appear with the later case. As a result, a new methodology involving a particular set of {\it cactus trees} is developed.
The main result can be stated as follows:
\begin{theorem}
\label{th1}
The numbers $M(n_1,n_2, n_3, N)$ of factorizations of a long cycle $\gamma_N=(1 2 \ldots N)$ into three permutations with $n_1, n_2, n_3$ cycles  verify:  
\begin{eqnarray}
\nonumber&& \sum_{n_1,n_2,n_3\geq 1}\frac{M(n_1,n_2, n_3, N)}{{N!}^2 }x_1^{n_1}x_2^{n_2}x_3^{n_3}=\sum _{p_1, p_2, p_3 \geq 1} \binomial{x_1}{p_1}\binomial{x_2}{p_2}\binomial{x_3}{p_3}\phantom{lllllllllalaalalalal}\\
\nonumber&&\phantom{lalalalala} \times \binomial{N-1}{p_3-1}\sum_{a \geq 0} \binomial{N-p_2}{p_1-1-a}\binomial{N-p_3}{a}\binomial{N-1-a}{N-p_2}
\end{eqnarray}
\end{theorem}

{\bf Outline of the paper and additional notations\\}
\indent The goal of this paper is to give a bijective proof of the result stated above.
In the following sections, we introduce the combinatorial objects used as the basis of our construction. Then we describe the bijective mapping and prove it is indeed one-to-one.\\
\indent In what follows, we note $P_n(A)$ (resp. $OP_n(A)$) the set of unordered (resp. ordered) subsets of set $A$ containing exactly $n$ elements.

\section{Maps, Constellations and Cacti}
Maps can be defined as $2-$cell decompositions of an oriented surface into a finite number of vertices ($0-$cells), edges ($1-$cells) and faces ($2-$cells) homeomorphic to open discs (see \cite{LZ} for more details about maps and their applications). They are defined up to an homeomorphism of the surface that preserves its orientation, the type of cells and incidences in the graph.  A map is bicolored if the vertices are colored black and white such that two adjacent vertices have different colors. Bicolored maps are also known as {\it hypermaps}.\\ 
\indent While $2$-factorizations of a permutation are easily interpreted as bicolored maps, $m$-factorizations ($m \geq 3$) can be represented as {\it $m$-constellations} (see \cite{LZ}). Within a topological point of view, $m$-constellations are specific maps with black and white faces such that all the black faces are $m$-gons and not adjacent to each other. The degree of the white faces is a multiple of $m$. Often (see \cite{BMS}), constellations are supposed to be planar maps (embedded in a surface of genus $0$). In this paper we suppose that they can be embedded in a surface of any genus. Besides, we consider only {\it rooted} constellations, i.e. constellations with a marked edge. We assume as well that within each $m$-gon the $m$ vertices are colored with $m$ distinct colors so that moving around the $m$-gons clockwise the vertex of color $i+1$ (modulo $m$) follows the vertex of color $i$. A constellation with only one white face is usually called a {\it cactus}. Figure 1 shows a $3$-cactus embedded in the sphere (genus $0$) and two $3$-cacti embedded in a surface of genus $1$.

\begin{proposition}
Rooted $m$-cacti with $N$ black $m$-gons are in bijection with $m$-tuples of permutations $\alpha_1$, $\alpha_2$, ..., $\alpha_m$ in $S_N$ such that $\alpha_1\alpha_2...\alpha_m = \gamma_N$
\end{proposition}

\begin{proof}[Proof(sketch)]
Let $C$ be a $m$-cactus with $N$ black faces of degree $m$ and one white face $W$ of degree $mN$ (as black faces have no adjacent edges, each edge belongs to exactly one black face as well as to the white face). We assume that the root edge links a vertex of color $1$ and a vertex of color $m$. Moving around the $mN$ vertices of $W$ counter-clockwise starting with the vertex of color $1$ of the root edge, the various edges connecting couples of vertices of color $1$ and $m$ are traversed exactly once. We denote $B_i$ the black face containing the $i'th$ such edge met. Then, we consider the permutations ${(\alpha_k)}_{1\leq k \leq m}$ in $S_N$ mapping $i$ to $j$ if and only if $B_i$ and $B_j$ have a vertex $v$ of color $k$ in common and $B_j$ follows $B_i$ in counter-clockwise order around $v$. According to the cyclic order of the vertices' colors, the permutation $\alpha_1\alpha_2...\alpha_m$ maps $i$ to $i+1$ (modulo $N$) for $1\leq i \leq m$. In other words, $\alpha_1\alpha_2...\alpha_m = \gamma_N$.\\
\noindent Conversely, given a $m$-tuple of permutations $\alpha_1\alpha_2...\alpha_m$ in $S_N$ such that $\alpha_1\alpha_2...\alpha_m = \gamma_N$ we can define a $m$-cactus following the construction rules below :
\begin{itemize}
\item Let $B_1$, $B_2$, ...,$B_N$ be $N$ $m$-gons whose vertices' colors are $1$, $2$, ..., $m$ in clockwise order. 
\item We use the permutations $\alpha_i$ to merge the polygons' vertices of the same color so that moving around a vertex $v$ of color $j$ counter-clockwise, $B_k$ is following $B_l$ if $\alpha_j(l)=k$.
\end{itemize}
Then, starting from $B_1$, follow the edges of the polygons linking the vertices of colors $1$ and $m$, $m$ and $m-1$, ..., $2$ and $1$, $1$ and $m$, etc. The edge $e'$ linking vertices $v'$ and $v"$ of colors $i$ and $i-1$ follows the edge $e$ linking $v$ of color $i+1$ and $v'$ if $e'$ follows $e$ in counter-clockwise order around $v'$. As  $\alpha_1\alpha_2...\alpha_m = \gamma_N$, we traverse all the $B_k$ and come back to $B_1$. This traversal defines the edges of the unique white face of degree $Nm$.   \end{proof}
 \begin{figure}[h]
 \label{cacspher}
\begin{center}
\includegraphics[width=40mm]{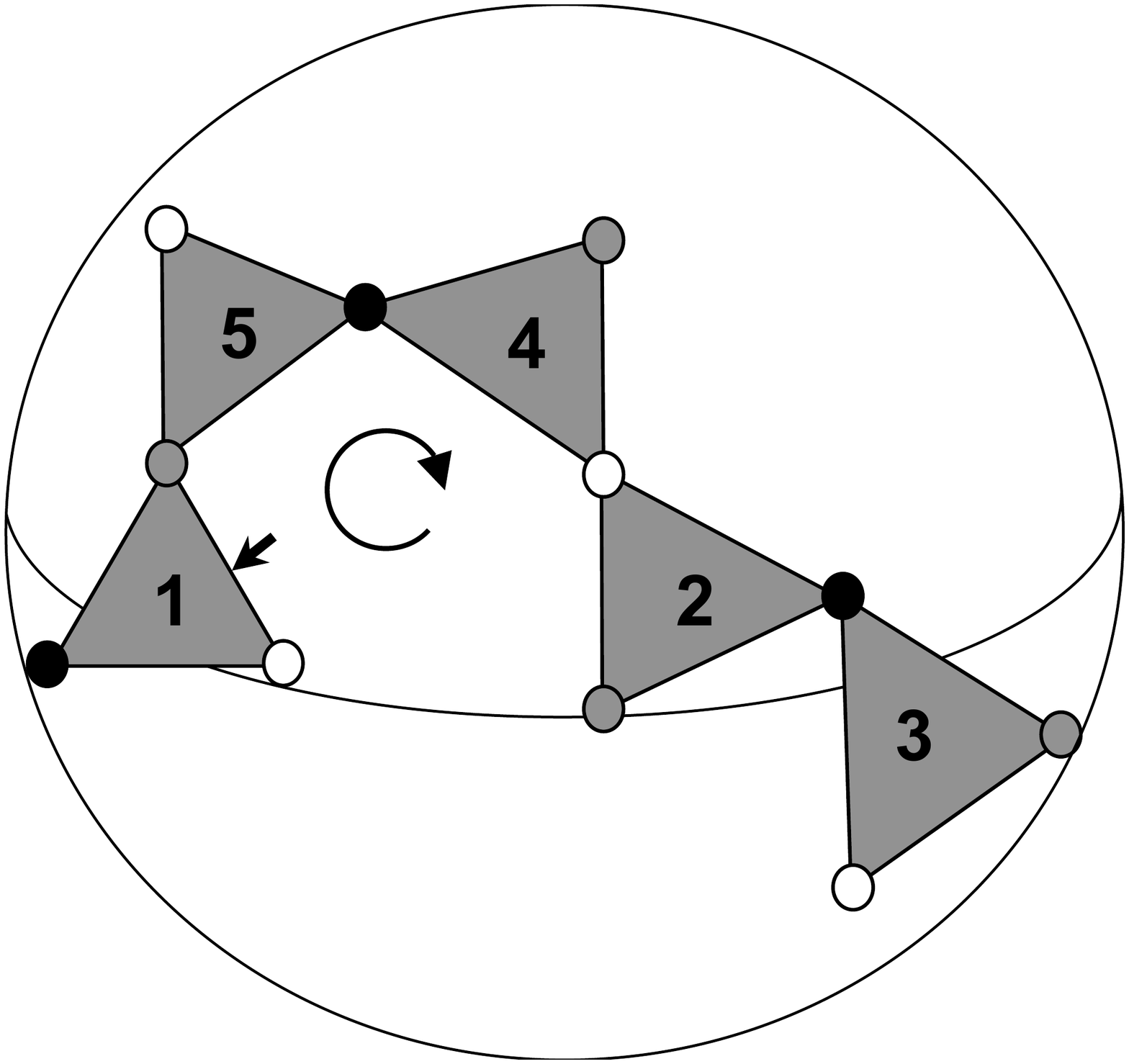}
\includegraphics[width=30mm]{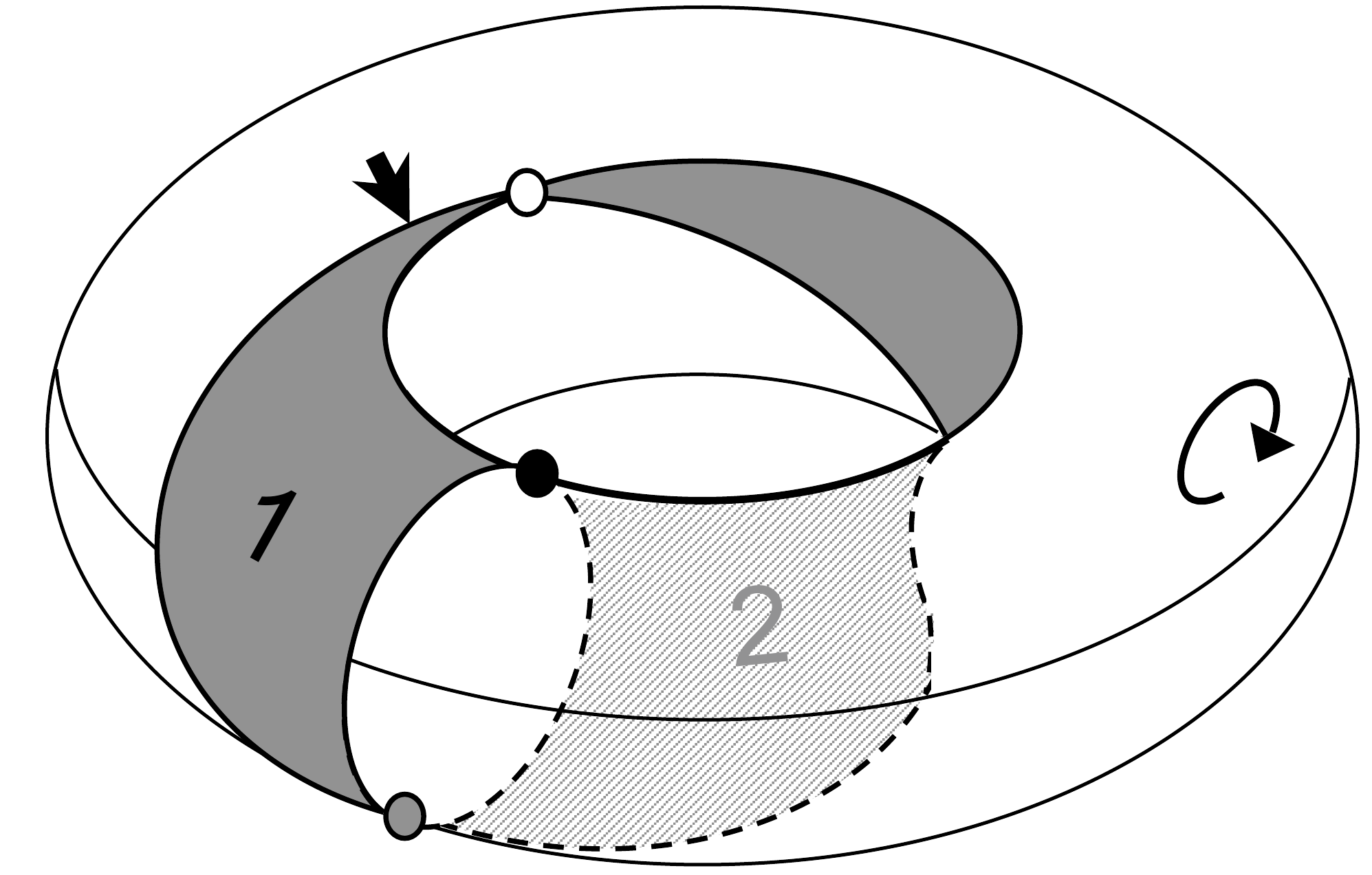}\hspace{5mm}
\includegraphics[width=30mm]{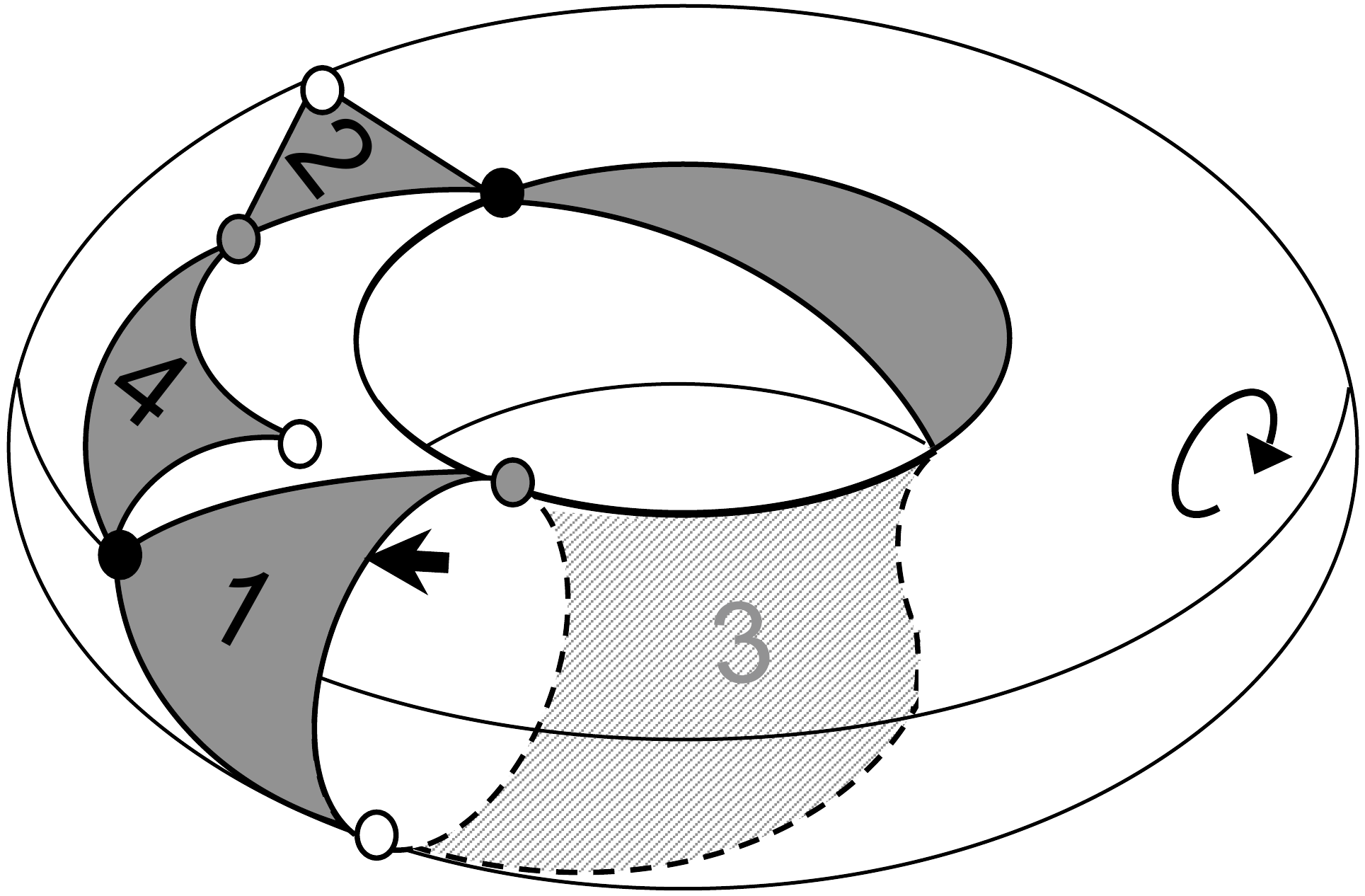}
\caption{A cactus embedded in the sphere and two examples of cacti embedded in a surface of genus 1}
\end{center} 
\end{figure}
\begin{example}
The cactus on the left hand side of figure 1 can be associated to the three permutations $\alpha_1 = (1)(24)(3)(5),\; \alpha_2 = (1)(23)(45),\; \alpha_3 = (15)(2)(3)(4)$.
\end{example}


\begin{example}
The cactus in the middle in figure 1 can be associated to the three permutations $\alpha_1 = (12),\; \alpha_2 = (12),\; \alpha_3 = (12)$.
The cactus on the right hand side corresponds to the three permutations $\alpha_1 = (13)(2)(4),\; \alpha_2 = (14)(23),\; \alpha_3 = (13)(24)$.
\end{example}



%
\section{Partitioned Cacti and Cactus Trees}
\subsection{Partitioned 3-Cacti}
As non recursive and non planar objects, cacti are very difficult to compute directly and so far no direct method
to enumerate general sets of cacti with given numbers of vertices of each color has been found. To overcome this issue, we introduce a new set of objects we named {\it Partitioned Cacti}. As shown below, there exists a simple relation between the cardinalities of sets of partitioned cacti and sets of traditional cacti. The point of this paper is to show that partitioned cacti can be interpreted via a bijective mapping as cactus trees and some simple combinatorial objects (permutations, ordered and unordered subsets of $[N]$). Being planar and recursive, cactus trees can be computed using the Lagrange theorem for implicit functions.
Formally, we have the following definition:\\
\begin{definition}
Let $ CC(p_{1}, p_{2}, p_{3}, N)$ be the set of $5$-tuples  $(\pi_{1}, \pi_{2}, \pi_{3}, \alpha_{1}, \alpha_{2})$ such that $\pi_{1}$,$\pi_{2}$ and $\pi_{3}$
are partitions of $\left[N\right]$
into $p_{1}$,  $p_{2}$ and  $p_{3}$ blocks and $\alpha_{1}$, $\alpha_{2}$ are permutations of $S_{N}$ such that each block of $\pi_{1}$ (resp. $\pi_{2}$, $\pi_{3}$)  is the union of cycles of $\alpha_{1}$ (resp. $\alpha_{2}$, $\alpha_{3}=\alpha_{2}^{-1}\circ\alpha_{1}^{-1}\circ\gamma_{N}$)
Any such $5$-uple is called a {\it partitioned $3$-cactus} with $N$ triangles, $p_{1}$ white,  $p_{2}$ black and  $p_{3}$ grey blocks.
\end{definition}


\begin{example}
\label{ex: pc1}
 As an example, we add partitions $\pi_{1}=\{ \pi_{1}^{(1)}, \pi_{1}^{(2)}\}, \pi_{2}=\{\pi_{2}^{(1)}, \pi_{2}^{(2)}\}$,
$\pi_{3}=\{\pi_{3}^{(1)}, \pi_{3}^{(2)}\}$, where
$\pi_{1}^{(1)}=\{2, 4, 5\},\,\,\,  \pi_{1}^{(2)}=\{1, 3\}, \; \pi_{2}^{(1)}=\{1, 2, 3\},\,\,\, \pi_{2}^{(2)}=\{4, 5\}\; \pi_{3}^{(1)}=\{3\},\,\,\, \pi_{3}^{(2)}=\{1, 2, 4, 5\}$, 
to the cactus of figure 1 in order to get the partitioned cactus $(\pi_{1}, \pi_{2}, \pi_{3}, \alpha_{1}, \alpha_{2})  \in {CC}({2, 2, 2, 5})$ depicted on figure~\ref{fig : ex}.
Similarly to \cite{SV}, we associate a particular shape to each of the blocks of the partitions.
\begin{figure}[h]
\begin{center}
\includegraphics[width=40mm]{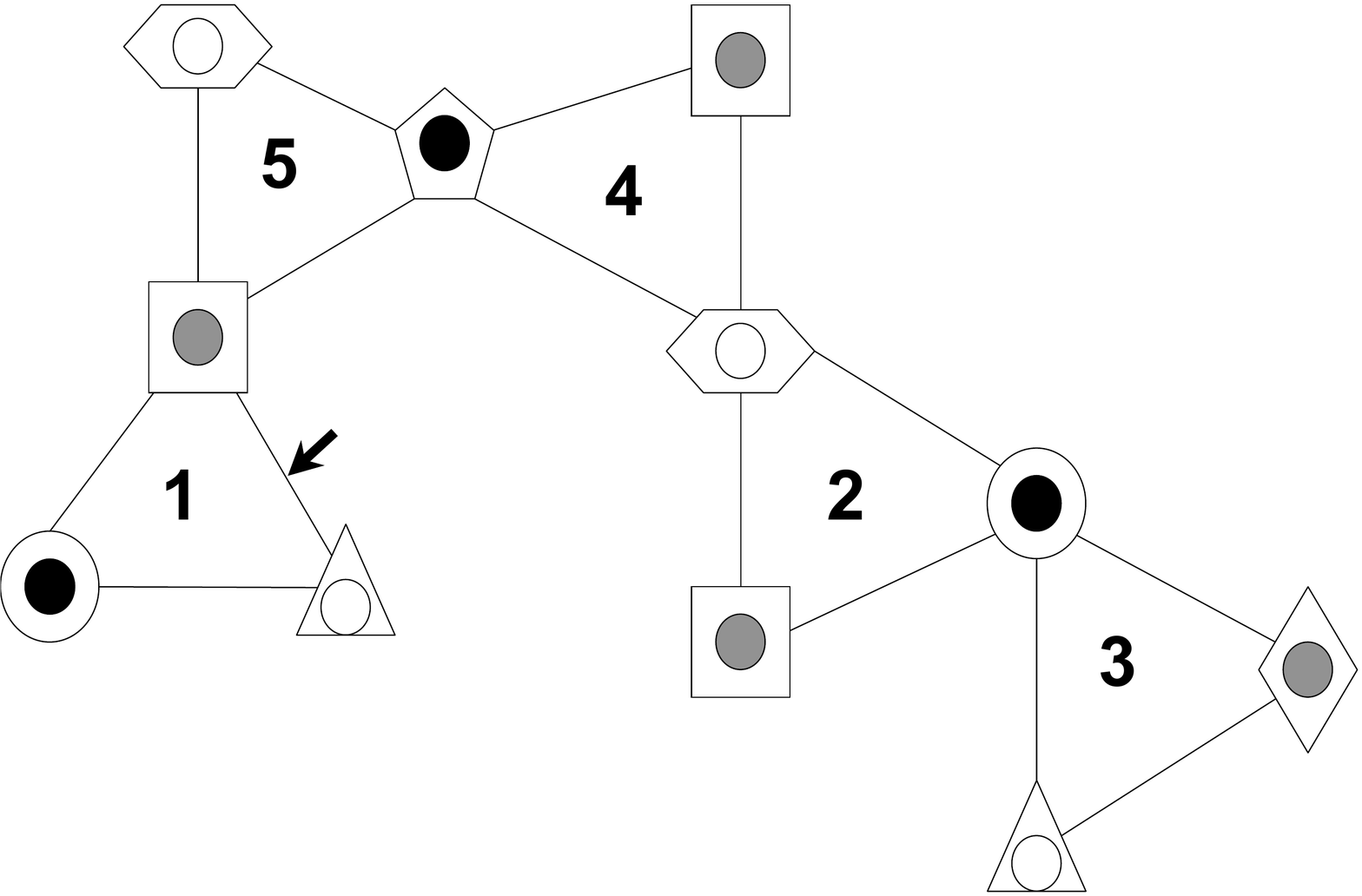}
\caption{Example of a Partitioned 3-Cactus}
\label{fig : ex}
\end{center} 
\end{figure}
\end{example}

\noindent {\bf{Link with cacti}}\\
Partitioned $3$-cacti and $3$-cacti are linked through the following formula :
\begin{eqnarray}
 \mid CC(p_1, p_2, p_3, N) \mid = \sum_{n_i  \geq p_i }S(n_1,p_1)S(n_2,p_2)S(n_3,p_3)M(n_1, n_2, n_3, N)
\end{eqnarray}
where $S(a,b)$, the Stirling number of the second kind, gives the number of partitions of a set of $a$ elements into $b$ nonempty, unordered sets. Using $\sum_{b=1}^{a}S(a,b)(x)_{b}$ $= x^{a}$ (see e.g~\cite{SW}) where the falling factorial $(x)_l = x(x-1)\ldots(x-l+1)$:
\begin{eqnarray}
\sum_{n_i \geq 1}M(n_1, n_2, n_3,N)x_1^{n_1}x_2^{n_2}x_3^{n_3} =\sum _{p_i \geq 1}\mid  CC(p_1, p_2, p_3, N)\mid (x_1)_{p_1}(x_2)_{p_2}(x_3)_{p_3}
\end{eqnarray}
In order to prove our main theorem, we focus on a bijective description of partitioned cacti. As stated above, our construction
relies on a particular set of cactus trees described afterwards in this section.

\subsection{White, Black and Grey Traversals of a Partitioned 3-Cacti and Last Passage}
\label{LastPassage}
We introduce the notion of white, black and grey traversals of a partitioned 3-cacti, a critical tool for the bijection.
Moving around the white face of the cactus according to the surface orientation and starting with the white to grey edge of the root, we define three new labelings of the triangles. We assign "white" (resp. "black", "grey") label $i$ (resp. $j$, $k$) to the triangle containing the $i$'th (resp. $j$'th, $k$'th)  edge linking a white and a grey (resp. a black and a white, a grey and a black) vertex. This three labeling of the faces will be called respectively white, black and grey traversal labeling. 
\begin{lemma}
\label{wbgt}
It can be shown that if $i$ is the initial label of a triangle, $i$ is as well its white traversal label, $\alpha_3^{-1}\alpha_2^{-1}(i)$ and
$\alpha_3^{-1}(i)$ are respectively its black and grey traversal labels.
\end{lemma}
\begin{example}
Figure \ref{wbgtraversals} depicts these three traversals for the partitioned cactus in example \ref{ex: pc1}.
\begin{figure}[h]
\begin{center}
\includegraphics[width=40mm]{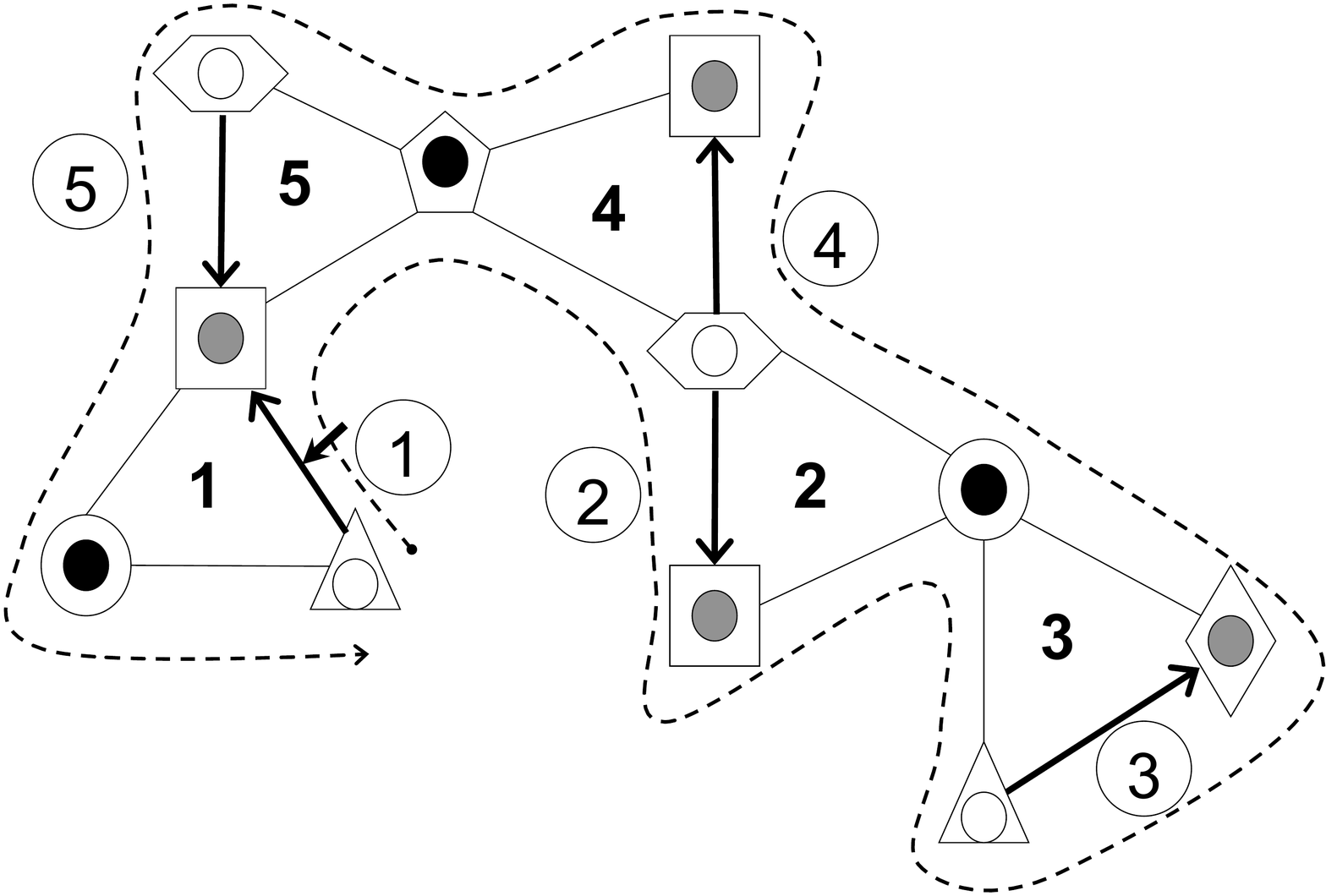}\
\includegraphics[width=40mm]{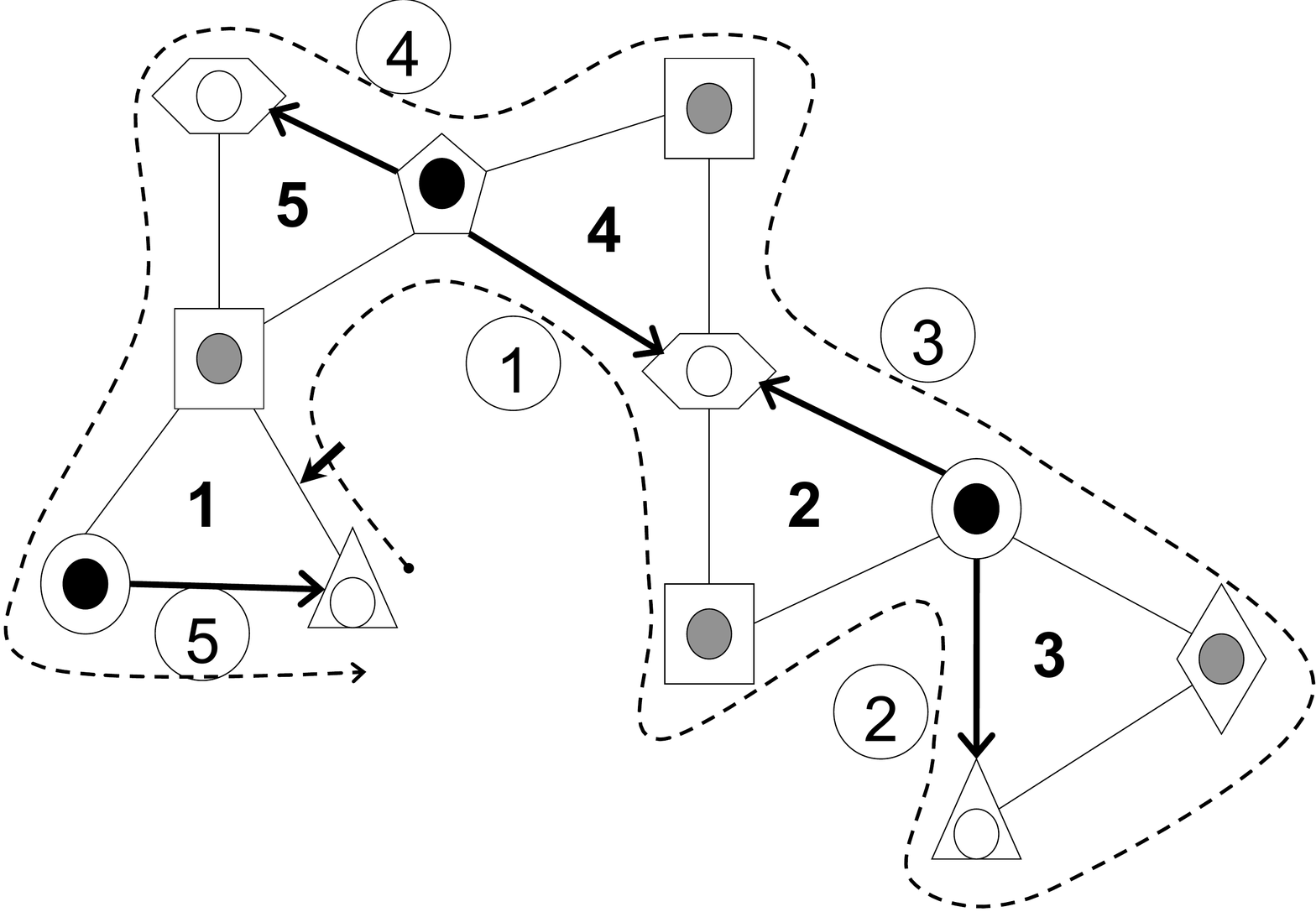}\
\includegraphics[width=40mm]{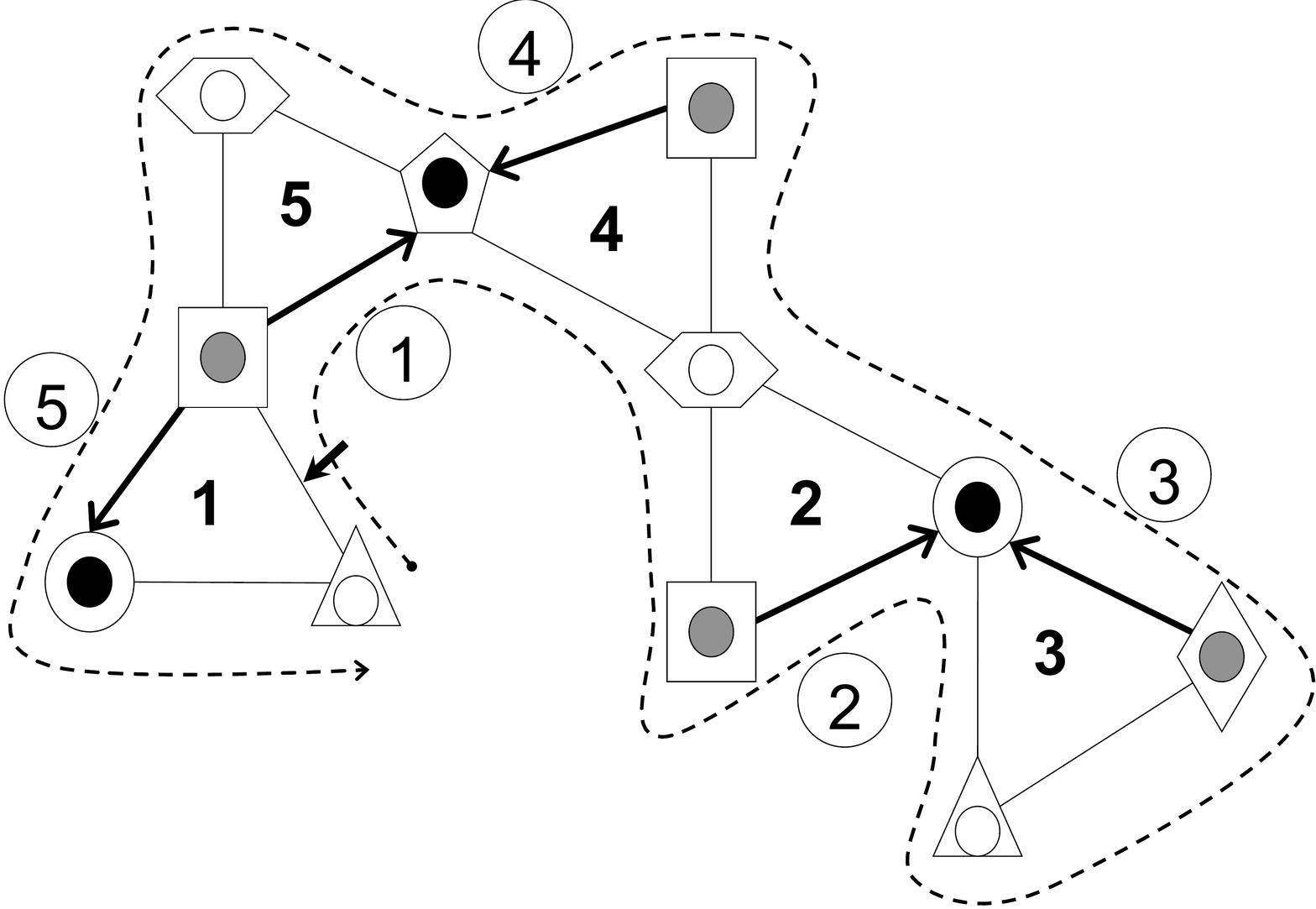}
\caption{White, black and grey traversals of a partitioned cactus}
\label{wbgtraversals}
\end{center} 
\end{figure}
\end{example}
Moreover, we define for any block of $\pi_1$ (resp. $\pi_2$, $\pi_3$), a {\it last passage} white (resp. black, grey) vertex. 
Let $\pi_{1}^{(1)}, \ldots,\pi_{1}^{(p_{1})}$, $\pi_{2}^{(1)}, \ldots,\pi_{2}^{(p_{2})}$
and $\pi_{3}^{(1)}, \ldots, \pi_{3}^{(p_{3})}$ be the blocks of the partitions $\pi_{1}$, $\pi_{2}$ and $\pi_{3}$ respectively, where
the indexing of the blocks is subject only to the condition that $1\in \pi_{1}^{(p_{1})}$.  
The last passage white (resp. black, grey) vertex of block $\pi_1^{(i)}$ (resp. $\pi_2^{(j)}$, $\pi_3^{(k)}$) is the one belonging to the triangle with maximum white (resp. black, grey) traversal label.
As an immediate consequence of lemma \ref{wbgt}, the (initial) indices of the triangles containing the last passage vertex of blocks $\pi_1^{(i)}$, $\pi_2^{(j)}$ and $\pi_3^{(k)}$ are respectively:
$$\max(\pi_1^{(i)}), \;\;\;\;\; \alpha_2\alpha_3(\max(\alpha_{3}^{-1}\alpha_2^{-1}(\pi_2^{(j)}))), \;\;\;\;\; \alpha_3(\max(\alpha_{3}^{-1}(\pi_3^{(k)})))$$
In what follows, we denote by $m_{1}^{(i)}$ the maximal element of $\pi_{1}^{(i)}$ ($ 1\leq i \leq p_{1}$),
by $m_{2}^{'(j)}$ the maximal element of the set $\alpha_{3}^{-1}\alpha_2^{-1}(\pi_{2}^{(j)})$ for $1\leq j\leq p_{2}$ (equal to $\alpha_{3}^{-1}(\pi_{2}^{(j)})$ as $\pi_2$ is stable by $\alpha_2$)\footnote{In general $m_{2}^{'(j)} \neq \max(\pi_{2}^{(j)})$} and by $m_{3}^{(k)}$ the maximal element of $\alpha_{3}^{-1}(\pi_{3}^{(k)}) = \pi_{3}^{(k)}$ for $1\leq k\leq p_{3}$.
\begin{example}
The initial indices of the last passage vertices in example \ref{ex: pc1} are 3 and 5 (white blocks), 5 and 1 (black blocks), 3 and 1 (grey blocks). Figure \ref{BLAST} shows how the last passage black vertex of the block represented by the circle is identified. 
\begin{figure}[h]
\begin{center}
\includegraphics[width=50mm]{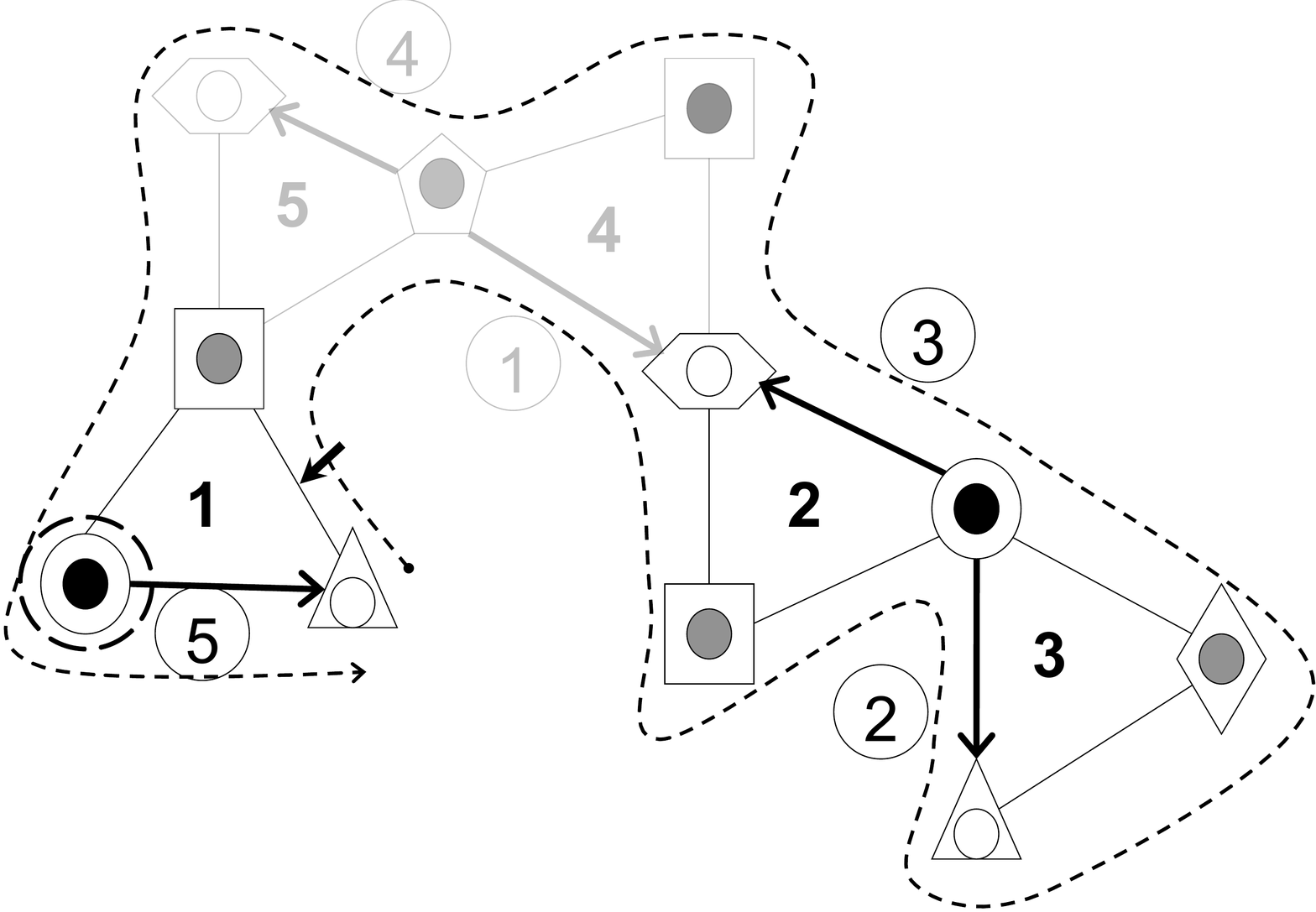}
\caption{Identification of the last passage vertex of a block of $\pi_2$}
\label{BLAST}
\end{center} 
\end{figure}
\end{example}

\subsection{Cactus Trees}

We focus on particular 3-colored cactus trees with two kinds of polygons: triangles and usual edges. More specifically we consider the sets $CT(p_{1}, p_{2}, p_{3}, a, b, c)$ of rooted cactus trees with $p_1$ white vertices, $p_2$ black vertices, $p_3$ grey vertices, $a$ triangles rooted in a grey vertex, $b$ triangles rooted in a white vertex, $c$ triangles rooted in a black vertex such that:

\begin{itemize}
\item {\tt } the root of the cactus tree is a white vertex
\item {\tt } a white vertex has black descending vertices and/or descending triangles rooted in this white vertex
\item {\tt } a black vertex has grey descending vertices and/or descending triangles rooted in this black vertex
\item {\tt } a grey vertex has white descending vertices and/or descending triangles rooted in this grey vertex
\item {\tt } triangles are composed of a white, a black and a grey vertex. Moving around triangles clockwise, these vertices are following each other in the white-black-grey-white cyclic order. As a direct consequence, triangles can
be seen as adding a "non-tree" edge linking a vertex and the rightmost descendant of one of its descendants. Note that edges cannot intersect triangles.
 \end{itemize}
\begin{example}
\label{ex : ex}
 Figure \ref{fig : ex ct} shows an example of a cactus tree of ${CT}({5, 4, 7, 1, 1, 2})$, 
 
\begin{figure}[h]
\begin{center}
\includegraphics[width=40mm]{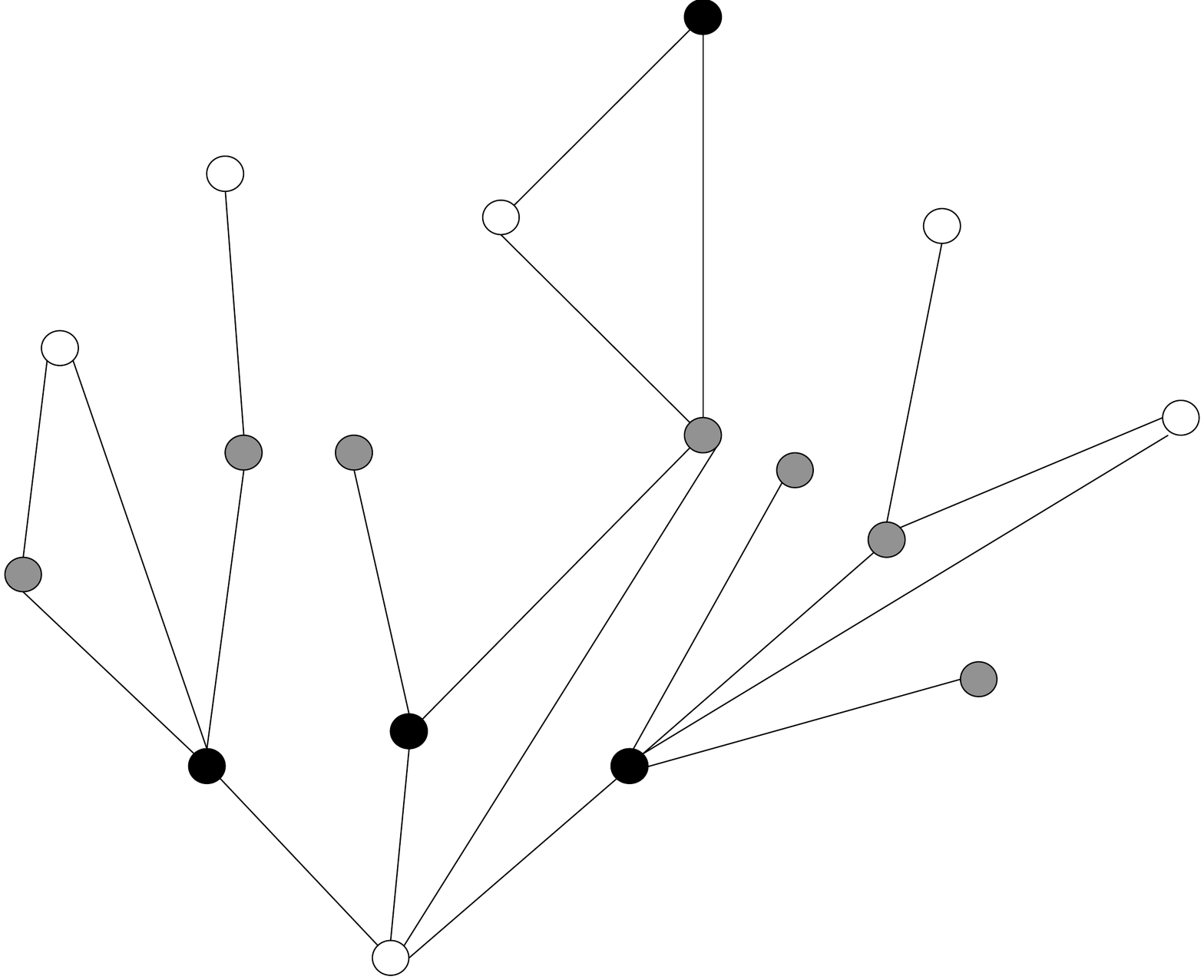}
\caption{Example of a Cactus Tree}
\label{fig : ex ct}
\end{center} 
\end{figure}

\end{example}

\begin{lemma} The cardinality of the considered set of cactus trees is :
\label{l : CT}
\footnotesize
\begin{eqnarray}
\nonumber |CT| = \frac{\left (a(b-p_3)+p_2p_3 \right)}{p_1p_2p_3}\binomial{p_1+p_2-1-a}{p_1-1, \,\,p_2-a-b}\binomial{p_2+p_3-1-b}{p_2-1, \,\,p_3-b-c}\binomial{p_1+p_3-2-c}{p_3-1, \,\,p_1-1-a-c}
\end{eqnarray}
\normalsize
\end{lemma} 

\begin{Proof}
The proof is postponed until appendix \ref{proof}
\end{Proof}

\section{A bijective description of partitioned cacti}
In this section we construct a bijective mapping $\Theta_{N, p_{1}, p_{2}, p_{3}}$.
Let 
$(\pi_{1}, \pi_{2}, \pi_{3}, \alpha_{1}, \alpha_{2})\in CC(p_{1}, p_{2}, p_{3}, N)$
be a partitioned $3-$cactus. Then we associate to it by $\Theta_{N, p_{1}, p_{2}, p_{3}}$ a 7-tuple 
$(\tau, S_{0}, S_{1}, S_{2}, \chi, \sigma_{1}, \sigma_{2})$
composed of 
\begin{itemize}
\item an ordered cactus tree $\tau \in CT(p_{1}, p_{2}, p_{3}, a, b, c)$ for a given triple $(a,b,c)$, 
\item unordered sets $S_{0} \in P_{p_1-1+p_2-a}([N])$,  $S_{1} \in P_{p_1-1+p_3-1-c}([N-1])$, $S_{2} \in P_{p_3+p_2-1-b}([N-1])$,
\item an ordered set $\chi \in OP_{p_3-b-c}([N])$ such that $S_0 \cap \chi = \emptyset$
\item and permutations $\sigma_{1} \in \Sigma_{N-p_1+1-p_3+c}$, $\sigma_{2}\in \Sigma_{N-p_2-p_3+b}$. 
\end{itemize}
The remaining of the section is devoted to detailed description of these objects and of how each of them is obtained 
from initial partitioned $3-$cactus.\\

\subsection{Cactus Tree}

{\bf Labeled $3$-colored tree}\\
First we build the last passage $3$-colored labeled tree $T$ with $p_{1}$ white, $p_{2}$ black and $p_{3}$ grey
vertices associated to the blocks of $\pi_1$, $\pi_2$ and $\pi_3$ such that white vertex $p_1$ associated with block $\pi_1^{(p_1)}$ is the root and ascendant/descendant relations are defined by the last passage vertices as follows. Black vertex $j$ associated to block $\pi_2^{(j)}$ is the descendants of white vertex $i$ associated to block $\pi_1^{(i)}$ if the last passage vertex of block $\pi_2^{(j)}$ belongs to the same triangle in the partitioned cactus as a white vertex of block $\pi_1^{(i)}$. If two black vertices $j_1$ and $j_2$ are descendant of the same white vertex $i$, then $j_1$ is on the left of $j_2$ if the white traversal label of the triangle containing the last passage vertex of $\pi_2^{(j_1)}$ is less than the one of $\pi_2^{(j_2)}$. Similar rules apply to define the grey descendants of a black vertex and the white descendants of a grey vertex. Figure \ref{2grey} illustrates the identification of the two grey descendants of the black vertex associated with a circle in example \ref{ex: pc1}, as well as their order.

\begin{figure}[h]
\begin{center}
\includegraphics[width=50mm]{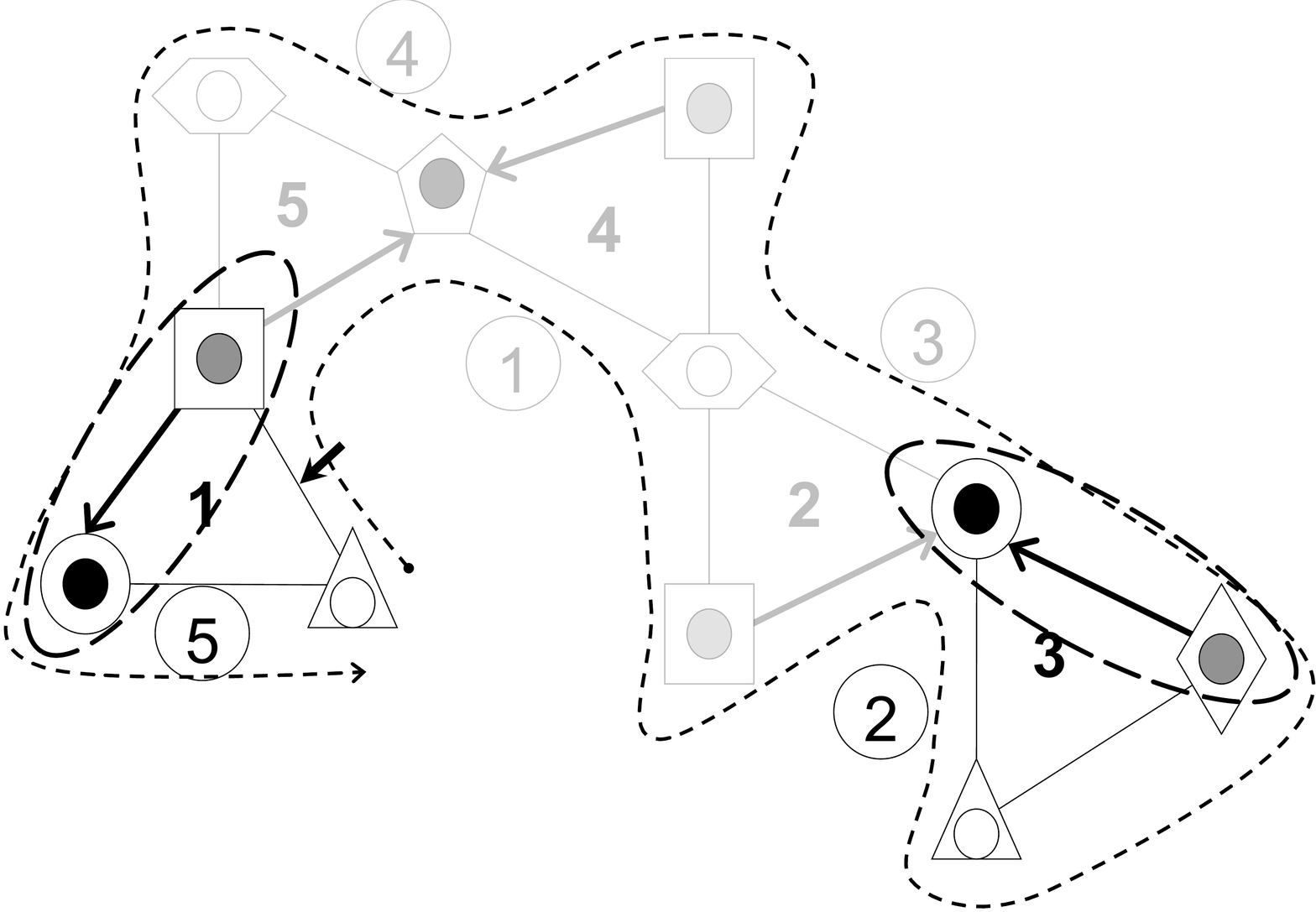}
\caption{Identification of the two grey descendant of the black vertex associated to circle shape}
\label{2grey}
\end{center} 
\end{figure}

More formally: 
\begin{itemize}
\item $(j_1,j_2,\ldots,j_s)$ is the ordered set of black descendants of white vertex $i$ if $\alpha_2\alpha_3(m_{2}^{'(j_1)}) < \alpha_2\alpha_3(m_{2}^{'(j_2)})<\ldots< \alpha_2\alpha_3(m_{2}^{'(j_s)})$ and $\alpha_2\alpha_3(m_{2}^{'(j_u)}) \in \pi_1^{(i)}$ for $1\leq u \leq s$.
\item $(k_1,k_2,\ldots,k_s)$ is the ordered set of grey descendants of black vertex $j$ if $\alpha_3^{-1}\alpha_2^{-1}\alpha_3(m_{3}^{(k_1)}) < \alpha_3^{-1}\alpha_2^{-1}\alpha_3(m_{3}^{(k_2)})<\ldots< \alpha_3^{-1}\alpha_2^{-1}\alpha_3(m_{3}^{(k_s)})$ and\\ $\alpha_3^{-1}\alpha_2^{-1}\alpha_3(m_{3}^{(k_u)}) \in \alpha_3^{-1}\alpha_2^{-1}(\pi_2^{(j)})$ for $1\leq u \leq s$.
\item $(i_1,i_2,\ldots,i_s)$ ($i_u \neq p_1$) is the ordered set of white descendants of grey vertex $k$ if $\alpha_3^{-1}(m_{1}^{(i_1)}) < \alpha_3^{-1}(m_{1}^{(i_2)})<\ldots< \alpha_3^{-1}(m_{1}^{(i_s)})$ and $\alpha_3^{-1}(m_{1}^{(i_u)}) \in \alpha_3^{-1}(\pi_3^{(k)})$ for $1\leq u \leq s$.
\end{itemize}
\begin{lemma} The labeled $3$-colored tree $T$ is correctly defined.\end{lemma}
\begin{proof}Assume non root white vertex $i_1$ is a descendant of grey vertex $k$, that is in its turn a descendant of black vertex $j$. Suppose now that $j$ is a descendant of white vertex $i_2$. Following our construction rules, we have
$m_{1}^{(i_1)} \in \pi_{3}^{(k)}$, so that $m_{1}^{(i_1)} \leq m_{3}^{(k)}$. Then  $m_{3}^{(k)} \leq m_{2}^{'(j)}$
as $m_{3}^{(k)} \in \alpha_3^{-1}(\pi_{2}^{(j)}) = \alpha_3^{-1}\circ \alpha_2^{-1}(\pi_{2}^{(j)})$($\pi_2$ is stable by $\alpha_2$). Finally, as $\alpha_2 \circ \alpha_3 (m_{2}^{'(j)}) \in \pi_1^{(i_2)}$,  $\gamma(m_{2}^{'(j)}) = \alpha_1 \circ \alpha_2 \circ \alpha_3 (m_{2}^{'(j)})\in \alpha_1(\pi_1^{(i_2)}) = \pi_1^{(i_2)}$ and $\gamma(m_{2}^{'(j)}) \leq m_{1}^{(i_2)}$. Two cases can occur :
\begin{itemize}
\item $m_{2}^{'(j)} \neq N$ then $\gamma(m_{2}^{'(j)}) = m_{2}^{'(j)}+1 \leq m_{1}^{(i_2)}$. As a direct consequence, $m_{1}^{(i_1)} <
m_{1}^{(i_2)}$
\item $m_{2}^{'(j)} = N$ and $\gamma(m_{2}^{'(j)}) = 1 \in \pi_1^{(i_2)}$. As a direct consequence, $i_2$ is the root vertex.
\end{itemize}
Using a similar argument, one can show that if black vertex $j_1$ (resp. grey vertex $k_1$) is on a descending branch of black vertex $j_2$ (resp. grey vertex $k_2$), then $m_{2}^{'(j_1)} < m_{2}^{'(j_2)}$ (resp. $m_{3}^{(k_1)} < m_{3}^{(k_2)}$).
Consequently, the white, black and grey vertices' labels on each branch of the graph are strictly increasing and the last vertex of the branch is always the root vertex (the block containing $1$). 
 \end{proof}
{\bf Cactus tree}\\
Remove  the labels from $T$ to obtain the {\it $3$-colored ordered tree} $t$. 
Next step is to construct the $3$-colored cactus tree $\tau$. To that purpose we draw a "non-tree" edge between the rightmost descendant and the ascendant of a {\it non root} vertex $v$ whenever the last passage vertices associated to $v$ and its rightmost descendant belong to the same triangle within the partitioned cactus. Clearly, this last condition can only be verified by the rightmost descendant. Indeed, last passage vertices associated with other descendants belong to triangles with strictly lower indices with respect to the  traversal of the color of $v$. This case happens for instance for the square and circle blocks on figure~\ref{triangle}. More formally, we have: 

\begin{figure}[h]
\begin{center}
\includegraphics[width=50mm]{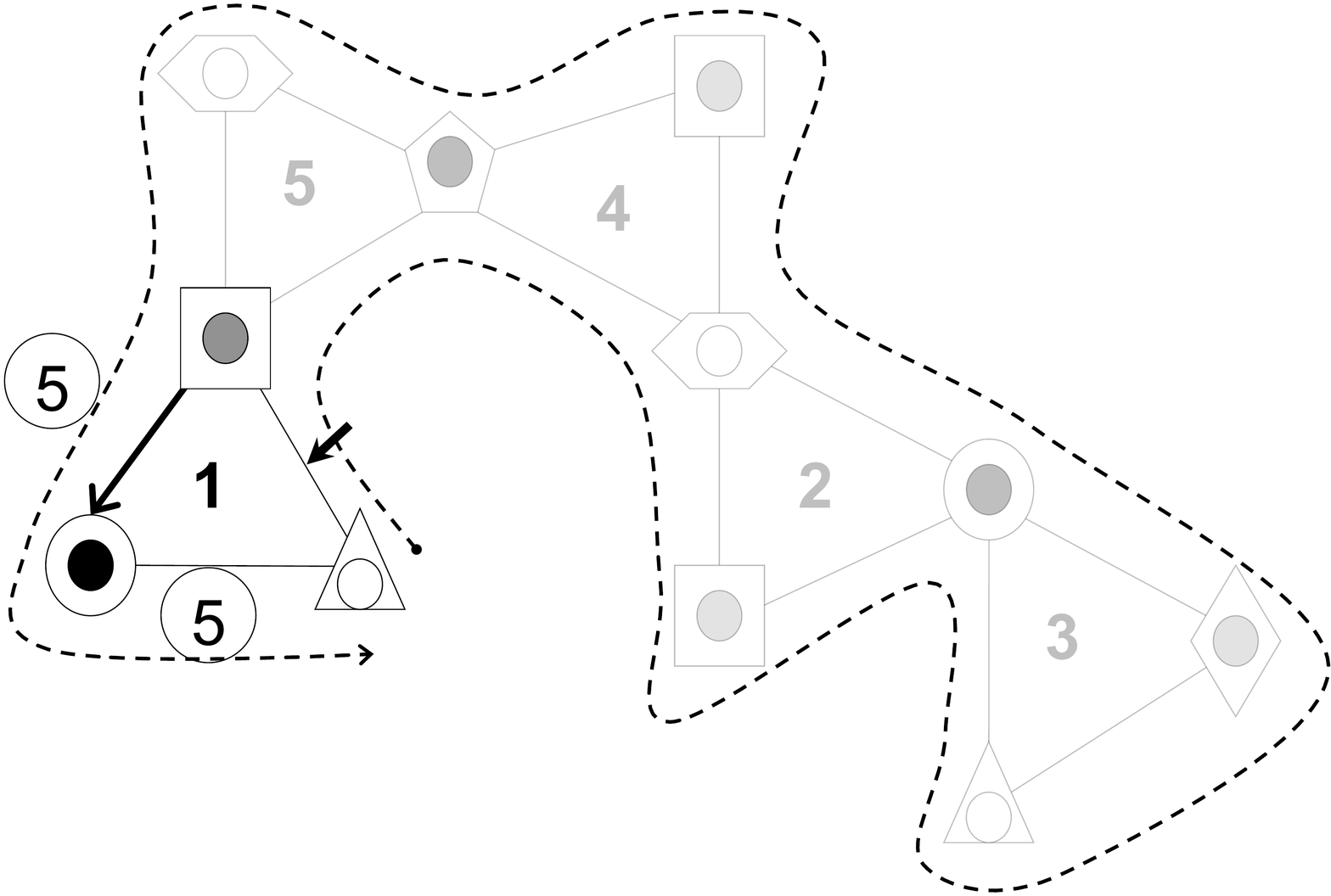}
\caption{Two last passage vertices belonging to the same triangle}
\label{triangle}
\end{center} 
\end{figure}

\begin{itemize}
\item a triangle rooted in a grey vertex when: $\alpha_2\circ \alpha_3(m_{2}^{'(j)}) =m_{1}^{(i)}$
\item a triangle rooted in a white vertex when: $\alpha_3(m_{3}^{(k)}) = \alpha_2\alpha_3(m_{2}^{'(j)})$
\item a triangle rooted in a black vertex when: $m_{1}^{(i)} = \alpha_3(m_{3}^{(k)})$
\end{itemize}

\begin{example}
\label{ex: t&ct}
Let us now apply the full procedure to example~\ref{ex: pc1}. Since
$m_{2}^{'(1)}=\max{\{2, 3, 5\}}=5 \in \alpha_3^{-1}\circ \alpha_2^{-1}(\pi_{1}^{(2)})=\{2, 5\}$, the black circle $1$
is a descendant of the root triangle $2$. As $m_{3}^{(1)}=3 \in \alpha_3^{-1}(\pi_{2}^{(1)})= \{2, 3, 5\}$ and
$m_{3}^{(2)}=5 \in \alpha_3^{-1}(\pi_{2}^{(1)})= \{2, 3, 5\}$, the grey rhombus $1$ and the grey square $2$ are both
descendants of the black circle $1$. Moreover, since 
$\alpha_3^{-1}\circ \alpha_2^{-1}\circ \alpha_3(m_{3}^{(1)})=2 <
\alpha_3^{-1}\circ \alpha_2^{-1}\circ \alpha_3(m_{3}^{(2)})=5$
the grey rhombus $1$ is to the left of the grey square $2$. The white hexagon $1$ is a descendant of the grey square $2$ since
$m_{1}^{(1)}=5 \in \pi_{3}^{(2)}=\{1, 2, 4, 5\}$ and the black pentagon $2$ is in its turn a descendant of the white hexagon $1$
as $m_{2}^{'(2)}=\max{\{1, 4\}}=4 \in \alpha_3^{-1}\circ \alpha_2^{-1}(\pi_{1}^{(1)})=\{1, 3, 4\}$. 
This is how we construct the tree $T$. Then by removing the labels and shapes from $T$ we get the tree $t$.  
As $\alpha_3^{-1}\circ \alpha_2^{-1}\circ \alpha_3(m_{3}^{(2)}) = m_{2}^{'(1)}=5$ 
we create a triangle rooted in the root vertex of $t$ connecting the root vertex of $t$ with its right grey vertex and, as
$\alpha_2\circ \alpha_3(m_{2}^{'(2)}) =m_{1}^{(i)}= 5$, we construct a triangle rooted in this grey vertex
connecting it with the top black vertex of $t$.
Finally, we obtain the cactus tree $\tau$ from the tree $t$ 
(see Figure~\ref{fig : t&ct}).
\end{example} 
From now to the end of this section, we will denote by 
$a$ the number of triangles rooted in grey vertices, by $b$ the number of triangles rooted in white vertices and 
by $c$ the number of triangles rooted in black vertices.

\begin{figure}[h]
\begin{center}
\includegraphics[width=50mm]{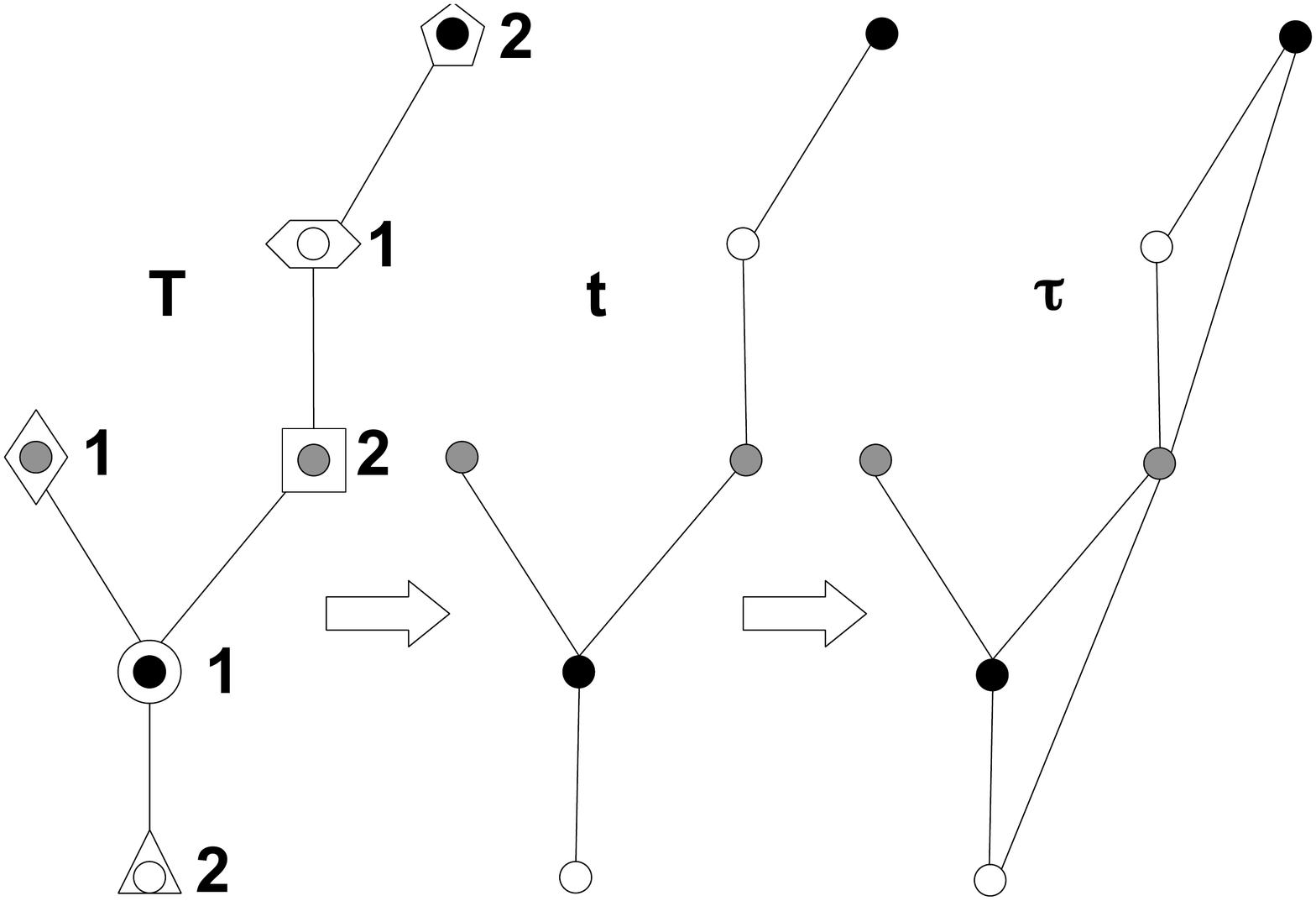}
\caption{Labeled Tree $T$, tree $t$ and Cactus Tree $\tau$}
\label{fig : t&ct}
\end{center} 
\end{figure}

\subsection{Permutations, ordered and unordered subsets}
{\bf Relabeled support sets}\\
\noindent {\it (i)} {\it Relabeling permutations.}\\
Using a similar approach as in \cite{SV}, we construct three relabeling permutations $\lambda_1,\lambda_2,\lambda_3$. 
First we consider the {\it reverse-labelled} tree $T'$ resulting from the labelling of $t$,
based on three independant reverse-labelling procedures for white, black and grey vertices.
The root is labelled $p_1$, the white vertices at level $4$ are labeled from right to left, beginning with $p_1-1$,  
proceeding by labelling  from right to left 
white vertices at level $7$ and all the other white levels until reaching the leftmost white vertex at the top white level labelled by $1$.
We proceed in the same way to label black and grey vertices. 
Next step consists in relabeling the blocks (white, black and grey) 
by using the new indices from $T'$. If a white  vertex is labeled $i$ in $T$ and $i'$ in $T'$, we note $\pi_1^{i'} = \pi_1^{(i)}$.
Black and grey blocks are relabeled in the similar way.
Let $\omega^{i}$, $\upsilon^{j}$, $\nu^{k}$ be the strings given by writing the elements of 
$\pi_{1}^{i}$, $\alpha_3^{-1}(\pi_{2}^{j})$, $\pi_{3}^{k}$ in increasing order.
Denote $\omega = \omega^{1}\dots \omega^{p_1}$, $\upsilon = \upsilon^{1}\dots \upsilon^{p_2}$,
$\nu = \nu^{1}\dots \nu^{p_3}$, concatenations of the strings defined above. 
We define $\lambda_1 \in \Sigma_{N}$ by setting $\omega$ the first line and $[N]$ 
the second line of the two-line representation of this permutation. Similarly,  we define $\lambda_2$ and $\lambda_3$.\\ 
%
%
\noindent {\it (ii)}{\it Support sets}\\
Now we  define the unordered set $S_{0} \subset [N]$ of size $\mid S_{0} \mid = p_1-1+p_2-a$ as 
$$S_{0}=\lambda_1 (\{m_{1}^{i}, \,\alpha_{2}\circ \alpha_{3}(m_{2}^{'j})\mid 1\leq i\leq p_{1}-1,\,1\leq j \leq p_{2}\}),$$
the set $S_{1} \subset [N]$ of size $\mid S_{1} \mid  = p_{1}-1+p_{3}-c$ as
$$S_{1}=\lambda_{3}( \{m_{3}^{k}, \, \alpha_{3}^{-1}(m_{1}^{i})\mid 1\leq k\leq p_{3},\, 1\leq i \leq p_{1}-1\})$$
the set $S_{2} \subset [N]$ of size $\mid S_{2} \mid  = p_{2}+p_{3}-b$ as
$$S_{2}=\lambda_{2}( \{m_{2}^{'j}, \, \alpha_{3}^{-1}\circ \alpha_{2}^{-1}\circ \alpha_{3}(m_{3}^{k})\mid 1\leq j\leq p_{2},\, 1\leq k \leq p_{3}\})$$
(note that $N$ is {\bf always} in $S_1$ and $S_2$ as $N=\lambda_3(m_{3}^{p_3})=\lambda_2(m_{2}^{p_2})$)\\
%
\begin{example}
\label{ex: relabeling}
Let us go back to our Example~\ref{ex: t&ct} and construct relabeling permutations 
and support sets. Using a reverse labeling of trees $T$ we can set:
\begin{eqnarray}
\nonumber
\pi_{1}^{1} =\pi_{1}^{(1)}, \; \pi_{1}^{2}=\pi_{1}^{(2)}, \; \pi_{2}^{1} =\pi_{2}^{(2)}, \; \pi_{2}^{2}=\pi_{2}^{(1)}, \; \pi_{3}^{1} =\pi_{3}^{(1)}, \; \pi_{3}^{2}=\pi_{3}^{(2)}     
\end{eqnarray}
Then the strings $\omega^{i}$, $\upsilon^{j}$ and $\nu^{k}$ are defined by:
\begin{eqnarray}
\nonumber
\omega^{1} = 245,\; \omega^{2}= 13,\; \upsilon^{1} = 14,\;  \upsilon^{2}= 235,\; \nu^{1} = 3,\;  \nu^{2}= 1245    
\end{eqnarray}
Let us construct the relabeling permutations $\lambda_{1}$, $\lambda_{2}$ and $\lambda_{3}$:

\footnotesize
\[ \lambda_{1}  = \left (\begin{array}{ccc}
\begin{array}{ccclccc}2&4&5\\1&2&3\end{array}
\left | \begin{array}{cclcc} 1&3\\4&5\end{array} \right.
\end{array} \right ) \lambda_{2} = \left (\begin{array}{cc}
\begin{array}{cclcc} 1&4\\1&2\end{array}
\left | \begin{array}{ccclccc}2&3&5\\3&4&5\end{array} \right.
\end{array} \right )   \lambda_{3} = \left (\begin{array}{cc}
\begin{array}{clc} 3\\1\end{array}
\left | \begin{array}{cccclcccc}1&2&4&5\\2&3&4&5\end{array} \right.
\end{array} \right )   \]
\normalsize
Then relabeled support sets are defined by:
\begin{eqnarray}
\nonumber S_{0} = \lambda_{1}(\{1, 5\}) = \{3, 4\}, \; S_{1} = \lambda_{3}(\{1, 3, 5\}) = \{1, 2, 5\}, \; S_{2} = \lambda_{2}(\{2, 4, 5\}) = \{2, 3, 5\}    
\end{eqnarray}
Relabeling permutations and support sets are pictured on figure \ref{figu : label}.
\end{example}

\begin{figure}
\begin{center}\includegraphics[width=40mm]{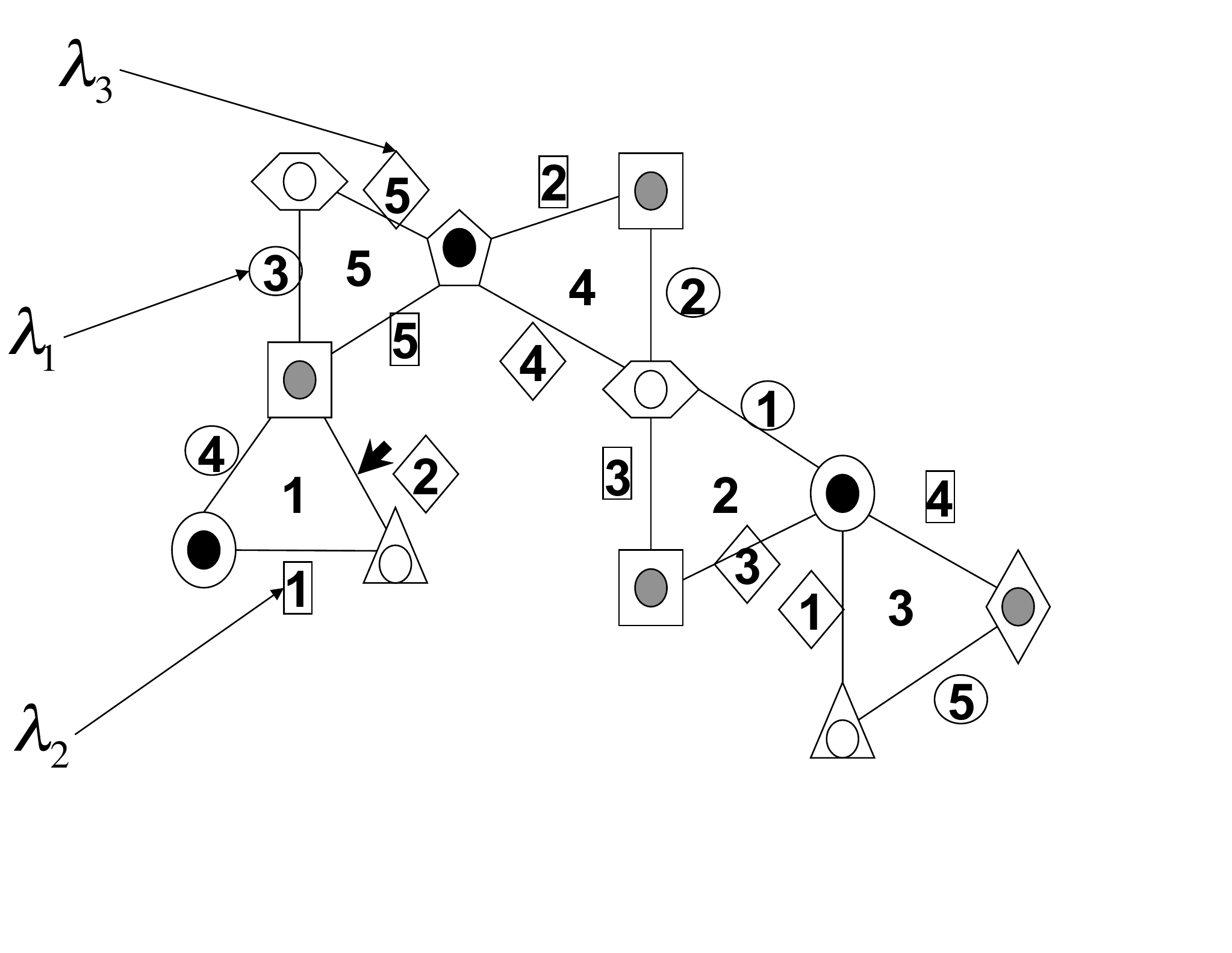}\
\includegraphics[width=38mm]{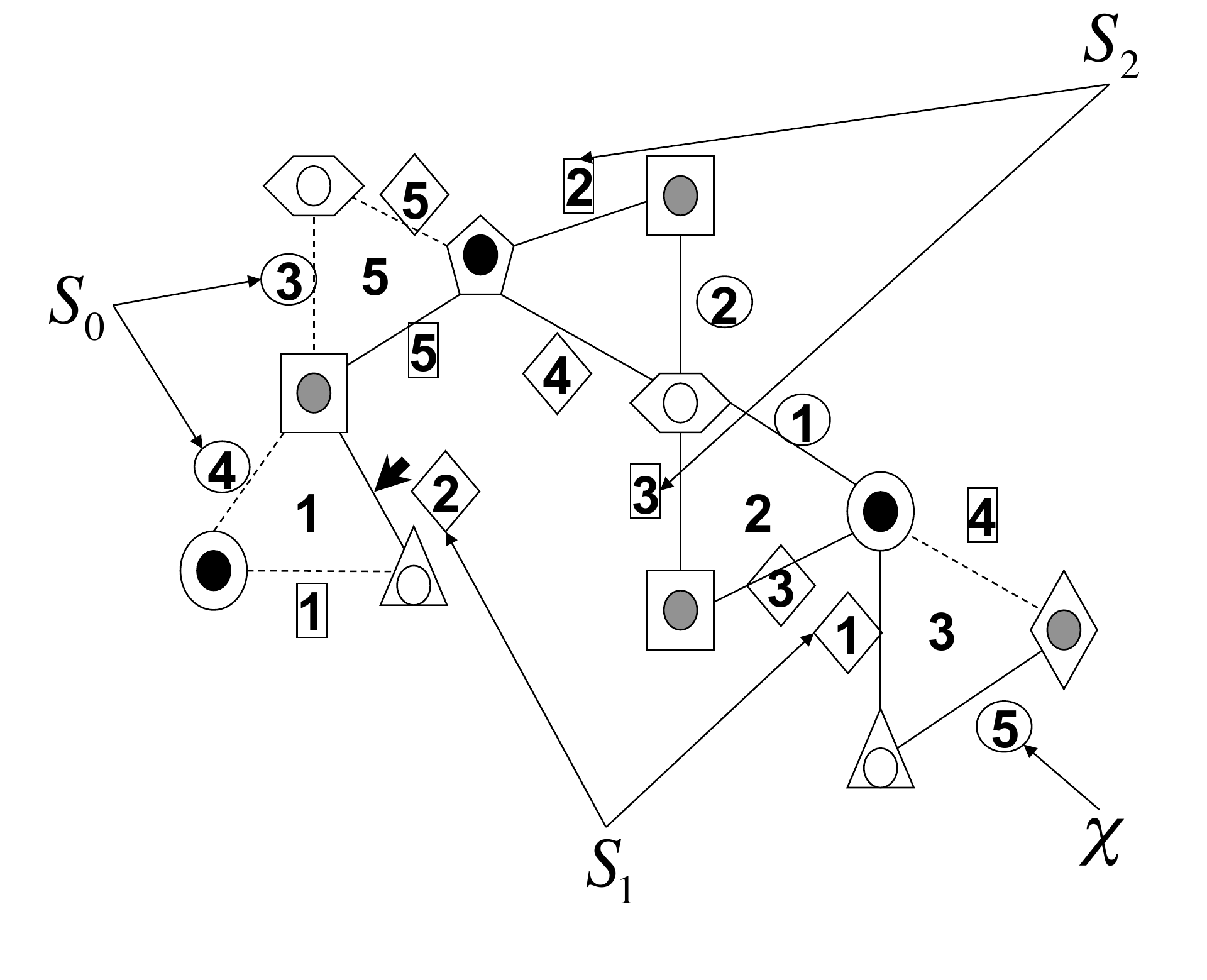}\
\includegraphics[width=40mm]{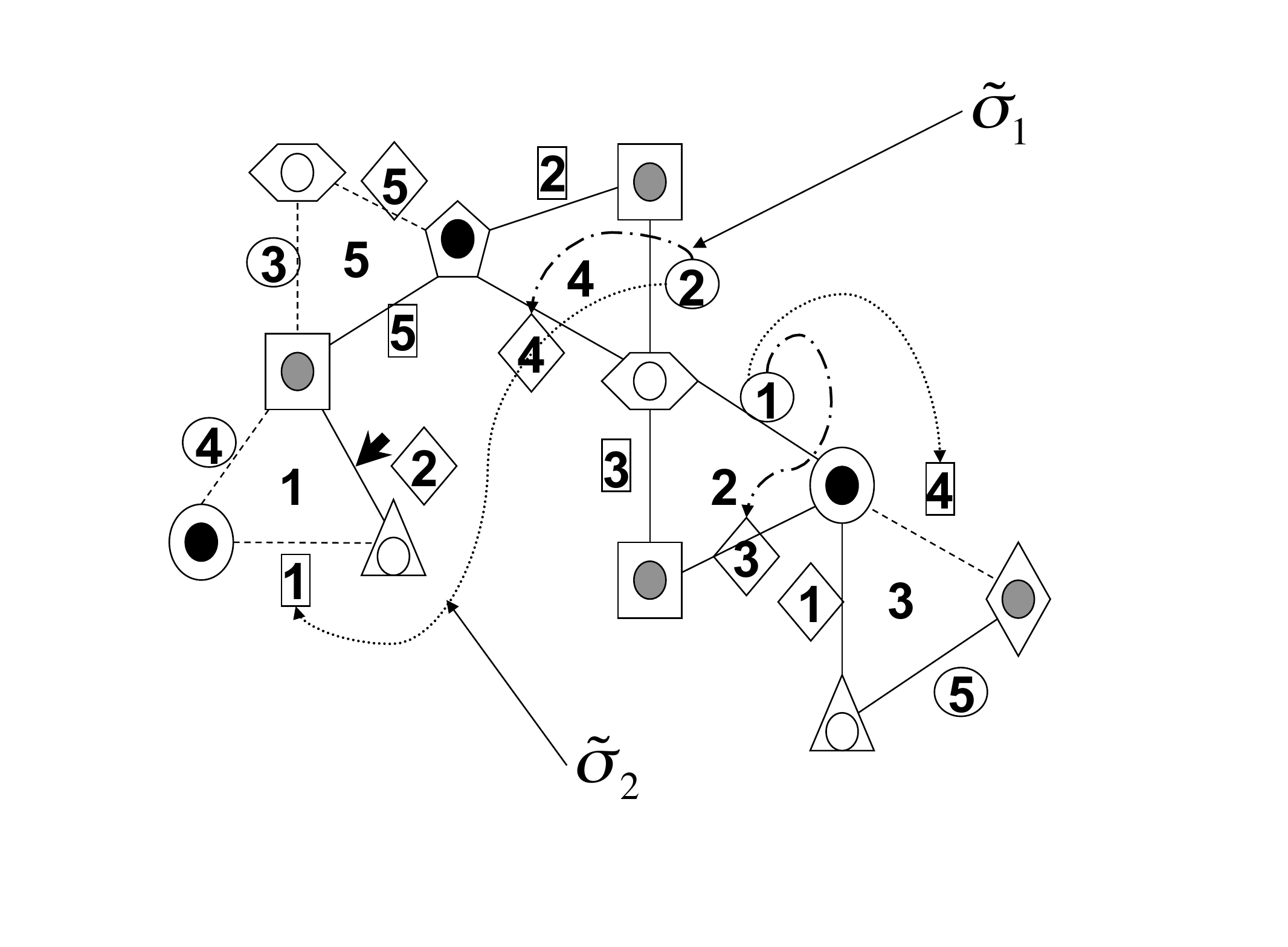}
\caption{Relabeling permutations, support sets and permutations}
\label{figu : label}
\end{center} 
\end{figure}
{\bf Relabeled ordered support complement set.}\\
We define an ordered set $$\chi=\left(\lambda_{1}\circ \alpha_{3}(m_{3}^{k}), k \leq p_{3} \mid \alpha_{3}(m_{3}^{k})\notin \{m_{1}^{i}, i\leq p_{1}-1,
\alpha_{2}\circ \alpha_{3}(m_{2}^{'j}), j \leq p_{2}\}\right)$$
of size $ \mid \chi \mid =p_{3}-b-c$. The relation $S_0 \cap \chi = \emptyset$ is clearly verified.\\
\begin{example}
\label{ex: chi}
The support complement set $\chi$ for our previous example~\ref{ex: relabeling} is $\chi=\left(\lambda_{1}(\{1, 3\}\setminus\{1\})\right) = \lambda_{1}(3) = 5$.
It is represented on figure \ref{figu : label}.
\end{example}

\noindent{\bf Permutations}\\
We need two additional objects to end our construction. We proceed in two steps.\\
{\it (i) Partial Permutations}\\
First we define the sets $E_1 = [N] \setminus \lambda_1\{m_{1}^{i}, \alpha_{3}(m_{3}^{k}) \mid 1\leq i\leq p_{1}-1,\,1\leq k \leq p_{3}\}$ and $E_2 = [N] \setminus \lambda_1\{\alpha_{2}\circ \alpha_{3}(m_{2}^{'j}), \alpha_{3}(m_{3}^{k}) \mid 1\leq j\leq p_{2},\,1\leq k \leq p_{3}\}$.
We construct partial permutations $\widetilde{\sigma_1}$ and $\widetilde{\sigma_2}$ defined respectively on $E_1$ and $E_2$ in the following way :
\begin{eqnarray}
\nonumber \widetilde{\sigma_1} :\phantom{lalalalaal}  E_1 \phantom{an}
&\longrightarrow& [N] \setminus S_1 \\
\nonumber u\phantom{an} \;& \longmapsto &\;\lambda_3\circ\alpha_3^{-1}\circ\lambda_1^{-1}(u)\\
\nonumber \phantom{\widetilde{\sigma_2} :}\phantom{lalalalaal}  \phantom{E} \phantom{an}\\
\nonumber \widetilde{\sigma_2} :\phantom{lalalalaal}  E_2 \phantom{an}
&\longrightarrow& [N] \setminus S_2 \\
\nonumber u\phantom{an} \;& \longmapsto &\;\lambda_2\circ\alpha_3^{-1}\circ\alpha_2^{-1}\circ\lambda_1^{-1}(u)
\end{eqnarray}







\noindent {\it (ii) Permutations}\\
For the second step, we define the ordered set $\overline{E_1}$ with all the elements of $E_1$ sorted in increasing order and $\rho_1$ the labeling function which associates to each element of $E_1$ its position index in $\overline{E_1}$. As a direct consequence, we have $\rho_1(E_1) = [N-p_1+1-p_3+c]$. Similarly we define $\rho_2$ to label the elements of $E_2$, $\rho_3$ to label the elements of
$[N] \setminus S_1$ and $\rho_4$ to label the elements of 
$[N] \setminus S_2$. We can now state the definition of permutations $\sigma_1$ and $\sigma_2$ as :
\begin{eqnarray}
\nonumber \sigma_1 :\phantom{[p_2 - b]}  [N-p_1+1-p_3+c] \phantom{an}
&\longrightarrow& [N-p_1+1-p_3+c] \\
\nonumber u\phantom{an} \;& \longmapsto &\; \rho_3\circ\widetilde{\sigma_1}\circ\rho_1^{-1}(u)\\
\nonumber \phantom{\widetilde{\sigma_2} :}\phantom{lalalalaal}  \phantom{E} \phantom{an}\\
\nonumber {\sigma_2} :\phantom{[p_1+1-c]}   [N-p_2 - p_3+b] \phantom{an}
&\longrightarrow& [N-p_2 - p_3+b] \\
\nonumber u\phantom{an} \;& \longmapsto &\;\rho_4\circ\widetilde{\sigma_2}\circ\rho_2^{-1}(u)
\end{eqnarray}
\begin{example}
\label{ex: permutations}
Let us continue the  previous Example~\ref{ex: chi}. The sets $E_{1}$ and $E_{2}$ are equal to
$ 
E_{1}= [4]\setminus \{1, 3\} = \{2, 4\}, E_{2}= [4]\setminus \{1, 3\} = \{2, 4\}
$.
Then the partial permutations $\widetilde{\sigma_{1}}$ and $\widetilde{\sigma_{2}}$  and the permutations $\sigma_{1}$ and $\sigma_{2}$ (see figure \ref{figu : label}) are:

\[  \widetilde{\sigma_{1}} = \left (\begin{array}{c}
\begin{array}{ccclccc}1&2\\3&4\end{array}
\end{array} \right )  \phantom{lyu} \widetilde{\sigma_{2}} = \left (\begin{array}{c}
\begin{array}{cclcc} 1&2\\5&4\end{array}
\end{array} \right )  \phantom{lyu}  \sigma_{1} = \left (\begin{array}{c}
\begin{array}{ccclccc}1&2\\1&2\end{array}
\end{array} \right )  \phantom{lyu} \sigma_{2} = \left (\begin{array}{c}
\begin{array}{cclcc} 1&2\\2&1\end{array}
\end{array} \right )   \]
\end{example}

\section{Derivation of the main formula}
Let  $I(p_1, p_2, p_3, N)$ be the set of 7-uple as defined in the previous section. We have:  
\begin{eqnarray}
\nonumber  I(p_1, p_2, p_3, N) &=& \left \{(\tau, S_{0}, S_{1}, S_{2}, \chi, \sigma_{1}, \sigma_{2}) \right.  \in\\
\nonumber && \bigcup_{a, b, c \geq 0} \left \{ CT(p_1, p_2, p_3, a, b, c) \times P_{p_1 - 1 + p_2 - a}([N]) \right. \\
\nonumber &&  \times P_{p_1 - 1 + p_3 - 1- c}([N -1]) \times P_{p_3 + p_2 - 1- b}([N - 1]) \\
 \nonumber && \left.  \times OP_{p_3 - b - c}([N])\times \Sigma_{N - p_1 + 1- p_3 + c}\times \Sigma_{N - p_2 - p_3 + b} \mid S_0 \cap \chi = \emptyset \right \}
\end{eqnarray}
We prove our main theorem by showing :
\begin{theorem} The mapping $\Theta$ defined by:
\label{th2}
\begin{eqnarray}
\nonumber   CC(p_{1}, p_{2}, p_{3}, N) \longrightarrow && I(p_1,p_2,p_3, N)\\
\nonumber (\pi_{1}, \pi_{2}, \pi_{3}, \alpha_{1}, \alpha_{2}) \longmapsto&& (\tau, S_{0}, S_{1}, S_{2}, \chi, \sigma_{1}, \sigma_{2})
\end{eqnarray}
is actually a bijection.
\end{theorem}
The proof of this theorem is detailed in the next section. Theorem \ref{th1} is a direct consequence of theorem \ref{th2}. Indeed, we have the following lemma:
\begin{lemma} The cardinality of the image set of $\Theta$ verifies:
\label{l : imth}
\begin{eqnarray}
 \mid I(p_1,p_2,p_3, N) \mid = \frac{{N!}^2}{p_1!p_2!p_3!}\binomial{N-1}{p_3-1} \sum_{a \geq 0} \binomial{N-p_2}{p_1-1-a}\binomial{N-p_3}{a}\binomial{N-1-a}{N-p_2}
\end{eqnarray}
\end{lemma}
\begin{Proof}
Cardinality of $CT(p_1,p_2,p_3,a,b,c)$ has already been stated in lemma \ref{l : CT}.
The numbers of ways to choose the unordered subsets $S_{0}$, $S_{1}$ and $S_{2},$ are defined by  binomial coefficients $$\binomial{N}{p_1-1+p_2-a}, \binomial{N-1}{p_1-1+p_3-1-c},  \binomial{N-1}{p_3+p_2-1-b}$$ (respectively).
Then, since $S_0 \cap \chi = \emptyset$, the number of ways to choose $\chi$ is equal to $$\binomial{N-p_1+1-p_2+a}{p_3-b-c}(p_3-b-c)!$$
The number of permutations $ \sigma_{1}, \sigma_{2}$ is equal to $(N-p_1+1-p_3+c)!, (N-p_2-p_3+b)!$\\
Combining everything and summing over $(a,b,c)$ gives:
\footnotesize
 \begin{eqnarray}
 \nonumber \mid I(p_1,p_2,p_3, N) \mid &=& \sum_{a,b,c \geq 0} \frac{\left (a(b-p_3)+p_2p_3 \right)}{p_1p_2p_3}\binomial{p_1+p_2-1-a}{p_1-1, \,\,p_2-a-b}\binomial{p_2+p_3-1-b}{p_2-1, \,\,p_3-b-c} \\
\nonumber &\times&  \binomial{p_1+p_3-2-c}{p_3-1, \,\,p_1-1-a-c}\binomial{N}{p_1-1+p_2-a}\binomial{N-1}{p_1-1+p_3-1-c}\\
\nonumber &\times& \binomial{N-1}{p_3+p_2-1-b}\binomial{N-p_1+1-p_2+a}{p_3-b-c}(p_3-b-c)!\\
\label{ee} &\times& (N-p_1+1-p_3+c)!(N-p_2-p_3+b)!
\end{eqnarray}
\normalsize
Equation (\ref{ee}) leaves room for a lot of simplifications on the binomial coefficients:
 \begin{eqnarray}
\nonumber \mid I(p_1,p_2,p_3, N) \mid &&= \frac{{(N-1)!}^2}{p_1!p_2!p_3!}\sum_{a,b,c \geq 0} \left (a(b-p_3)+p_2p_3 \right)\\
&& \times \binomial{N}{a,\,\,b,\,\,c,\,\,p_1-1-a-c, \,\,p_2-a-b,\,\,p_3-b-c}
\end{eqnarray}
where the last element on the right hand side of the equation is a multinomial coefficient.
The computation is conducted to the final result by arranging properly the terms depending on $b$ and $c$, and summing over these two parameters with the help of Vandermonde's convolution. 
\end{Proof}
\begin{remark}
The formula above can be symmetrized by noticing that $a(b-p_3) +p_2p_3 = ab + (p_2-a-b)(p_3-b-c)+bp_3+b(p_2-a-b)+c(p_2-a-b)$  and doing:
\footnotesize
\begin{eqnarray*}
&&\sum_{a,b,c}ab\binomial{N}{a,b,c,p_1-1-a-c,p_2-a-b,p_3-b-c}\\
&&\phantom{lalaaaaaaaaaaaaaalalala}=\sum_{a,b,c}aN\binomial{N-1}{a,b-1,c,p_1-1-a-c,p_2-a-b,p_3-b-c}\\
&&\phantom{lalaaaaaaaaaaaaaalalala}=\sum_{a,b,c}aN\binomial{N-1}{a,b,c,p_1-1-a-c,p_2-a-b-1,p_3-b-1-c}
\end{eqnarray*}
\normalsize
\noindent Furthermore:
\footnotesize
\begin{eqnarray*}
&&\sum_{a,b,c}\left [(p_2-a-b)(p_3-b-c)+bp_3+b(p_2-a-b)+c(p_2-a-b)\right]\\
&&\times\binomial{N}{a,b,c,p_1-1-a-c,p_2-a-b,p_3-b-c}\\
 &&=  \sum_{a,b,c}N(N-a)\binomial{N-1}{a,b,c,p_1-1-a-c,p_2-1-a-b,p_3-1-b-c}
\end{eqnarray*}
\normalsize
\noindent Putting everything together, we have:
 \begin{eqnarray}
 |I| = \frac{{N!}^2}{p_1!p_2!p_3!}\sum_{a,b,c \geq 0}\binomial{N-1}{a,b,c,p_1-1-a-c,p_2-1-a-b, p_3-1-b-c}
\end{eqnarray}
This symmetric version is a direct consequence of Jackson's formula in \cite{J}.
\end{remark}
\section{Proof of the bijection}
\label{proof}
{\bf Injectivity}\\
For the first step of the proof we focus on injectivity of $\Theta$. Let $(\tau, S_{0}, S_{1}, S_{2}, \chi, \sigma_{1}, \sigma_{2})$ be the image by $\Theta$ of $(\pi_{1}, \pi_{2}, \pi_{3}, \alpha_{1}, \alpha_{2})\in CC(p_{1}, p_{2}, p_{3}, N)$. Our aim is to show that $(\pi_{1}, \pi_{2}, \pi_{3}, \alpha_{1}, \alpha_{2})$ is uniquely determined by  $(\tau, S_{0}, S_{1}, S_{2}, \chi, \sigma_{1}, \sigma_{2})$. We proceed in the following way.\\
\noindent First we note that $S_0$ and $\tau$ determines  $\lambda_1 (\{m_{1}^{i}, \,\alpha_{2}\circ \alpha_{3}(m_{2}^{'j})\mid 1\leq i\leq p_{1}-1,\,1\leq j \leq p_{2}\})$ since $S_0$ gives the set of values of all these elements (by construction) and $\tau$ gives the relations of order on them (including equalities). Indeed, according to the construction of $\tau$ and $\lambda_1$, the $\lambda_1(m_1^{l})$'s are sorted in increasing order with respect to the reverse label order in $\tau$, the $\lambda_1 (\alpha_{2}\circ \alpha_{3}(m_{2}^{'l}))$'s associated with black vertices descending from the same white vertex are sorted in increasing order from left to right. Assume the $j'$th black vertex in the reverse labeling of $\tau$ is the descendant of the $i'th$ white vertex. Necessarily $\lambda_1 (\alpha_{2}\circ \alpha_{3}(m_{2}^{'j})) \in \lambda_1 (\pi_1^i)$. As a consequence we have one of the following three cases: $\lambda_1 (\alpha_{2}\circ \alpha_{3}(m_{2}^{'j})) > \lambda_1(m_{1}^{l}), \mbox{ if } l<i, $ $\lambda_1 (\alpha_{2}\circ \alpha_{3}(m_{2}^{'j})) < \lambda_1(m_{1}^{l}) \mbox{ if } i<l $ or $\lambda_1 (\alpha_{2}\circ \alpha_{3}(m_{2}^{'j})) = \lambda_1(m_{1}^{i})$ if the $j'$th black vertex and the $i'$th white vertex belong to the same triangle rooted in grey vertex.\\
\indent Similarly the sets $\lambda_{3}( \{m_{3}^{k}, \, \alpha_{3}^{-1}(m_{1}^{i})\mid 1\leq k\leq p_{3},\, 1\leq i \leq p_{1}-1\})$ and $\lambda_{2}( \{m_{2}^{'j}, \, \alpha_{3}^{-1}\circ \alpha_{2}^{-1}\circ \alpha_{3}(m_{3}^{k})\mid 1\leq j\leq p_{2},\, 1\leq k \leq p_{3}\})$ are uniquely determined.\\ 
\indent Looking at the triangles rooted in a white or a black vertex within $\tau$ we can determine
 $\{\lambda_1(\alpha_3(m_3^k)) \mid \lambda_1(\alpha_3(m_3^k)) \in S_0\}$. The complementary ordered set $\chi$ uniquely determines the $\{\lambda_1(\alpha_3(m_3^k)) \mid \lambda_1(\alpha_3(m_3^k)) \notin S_0\}$.

\begin{example}
Assume $(\tau, S_{0}, S_{1}, S_{2}, \chi, \sigma_{1}, \sigma_{2})$ is given by figure \ref{ex rec}.
\begin{figure}
\begin{center}\includegraphics[width=60mm]{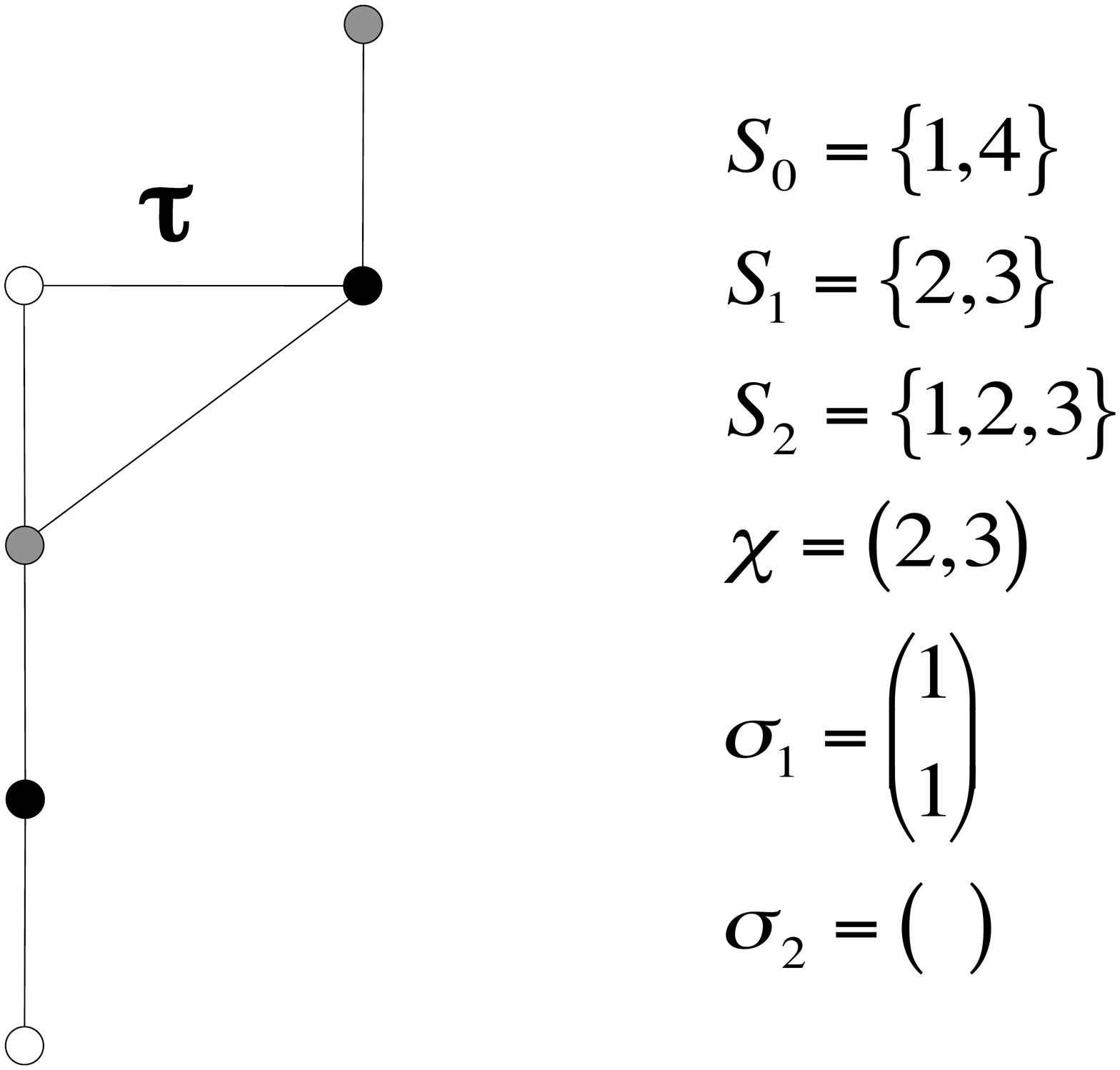}
\caption{Example of a cactus tree, sets and permutations}
\label{ex rec}
\end{center} 
\end{figure}
Following the construction rules stated above, we have necessarily: 
\begin{itemize}
\label{recons}
\item[(i)] $1=\lambda_1((\alpha_{2}\circ \alpha_{3}(m_{2}^{'1}))=\lambda_1(m_{1}^{1})$ since
the first black and first white vertices of the cactus tree belong to the same triangle and
$1$ is the smallest in $S_0$
\item[(ii)] $\lambda_1((\alpha_{2}\circ \alpha_{3}(m_{2}^{'2})) = 4$
\item[(iii)] $\lambda_3(m_{3}^{1})=2$, $\lambda_3(\alpha_{3}^{-1}(m_{1}^{1}))=3$ and $\lambda_3(m_{3}^{2})=4$
\item[(iv)] $\lambda_{2}(\alpha_{3}^{-1}\circ \alpha_{2}^{-1}\circ \alpha_{3}(m_{3}^{1})) = 1$, $\lambda_2(m_{2}^{'1}) = 2$, $\lambda_{2}(\alpha_{3}^{-1}\circ \alpha_{2}^{-1}\circ \alpha_{3}(m_{3}^{2})) = 3$ and $\lambda_2(m_{2}^{'2})=4$
\item[(v)] $\lambda_1(\alpha_3(m_3^1)) = 2$, $\lambda_1(\alpha_3(m_3^2)) = 3$ 
\end{itemize}
\end{example}

\indent The following step is to notice that elements reconstructed so far uniquely determine the supports of the partial permutations $\widetilde{\sigma_1}$ and $\widetilde{\sigma_2}$ and the partial permutations themselves. Then we define extensions $\overline{\sigma_1} = \lambda_3\circ\alpha_3^{-1}\circ\lambda_1^{-1}$ and $\overline{\sigma_2} = \lambda_2\circ\alpha_3^{-1}\circ\alpha_2^{-1}\circ\lambda_1^{-1}$
of $\widetilde{\sigma_1}$
and $\widetilde{\sigma_2}$ on the whole set $[N]$ by setting 
$\overline{\sigma_1}(i) = \widetilde{\sigma_1}(i)$ if $i$ belongs to the support of $\widetilde{\sigma_1}$, and $\lambda_3(m_3^k) = \overline{\sigma_1} (\lambda_1(\alpha_3(m_3^k)))\mbox{ or } \lambda_3(\alpha_3^{-1}(m_1^i)) = \overline{\sigma_1} (\lambda_1(m_1^i))$ otherwise (we have a similar construction for $\overline{\sigma_2}$). Now the partition $\lambda_1(\pi_1)$ is uniquely determined as well since
$\lambda_1(\pi_{1}^{1})=[\lambda_1(m_{1}^{1})]$, $\lambda_1(\pi_{1}^{i})=[\lambda_1(m_{1}^{i})]\setminus[\lambda_1(m_{1}^{i-1})] \mbox{ for } 2\leq i\leq p_1$.
Similar reconstruction applies to  $\lambda_2(\alpha_3^{-1}\pi_2)$ and $\lambda_3(\pi_3)$. By applying $\overline{\sigma_1}^{-1}$ (resp. $\overline{\sigma_2}^{-1}$) to $\lambda_3(\pi_3)$ (resp. $\lambda_2(\alpha_3^{-1}\pi_2)$) we determine uniquely $\lambda_1(\pi_3)$ (resp.  $\lambda_1(\pi_2)$).

\begin{example}
Following example \ref{recons}, we find that the domain of $\widetilde{\sigma_1}$ is $[4]\setminus\{1,2,3\} = \{4\}$. Its range is $[3]\setminus\{2,3\} = \{1\}$. The construction of $\overline{\sigma}_1$ and $\overline{\sigma}_2$ yields:
\[  \overline{\sigma}_{1} = \left (\begin{array}{c}
\begin{array}{cccclcccc}1&2&3&4\\ 3&2&4&1\end{array}
\end{array} \right )  \phantom{lyu} \overline{\sigma}_{2} = \left (\begin{array}{c}
\begin{array}{cccclcccc} 1&2&3&4\\2&1&3&4\end{array}
\end{array} \right )     \]
\noindent  As a next step, we have:
\begin{eqnarray}
\nonumber &\lambda_1(\pi_{1})=\{\{1\};\{2,3,4\}\}, \lambda_3(\pi_{3})=\{\{1,2\};\{3,4\}\}, \lambda_1(\pi_{3})=\{\{2,4\};\{1,3\}\}&\\
\nonumber &\lambda_2(\alpha_3^{-1}\pi_{2})=\{\{1,2\};\{3,4\}\}, \lambda_1(\pi_{2})=\{\{1,2\};\{3,4\}\}&
\end{eqnarray}

\noindent Figure \ref{bb} depicts the resulting partitioned cactus.

\begin{figure}[h]
\begin{center}
\includegraphics[width=35mm]{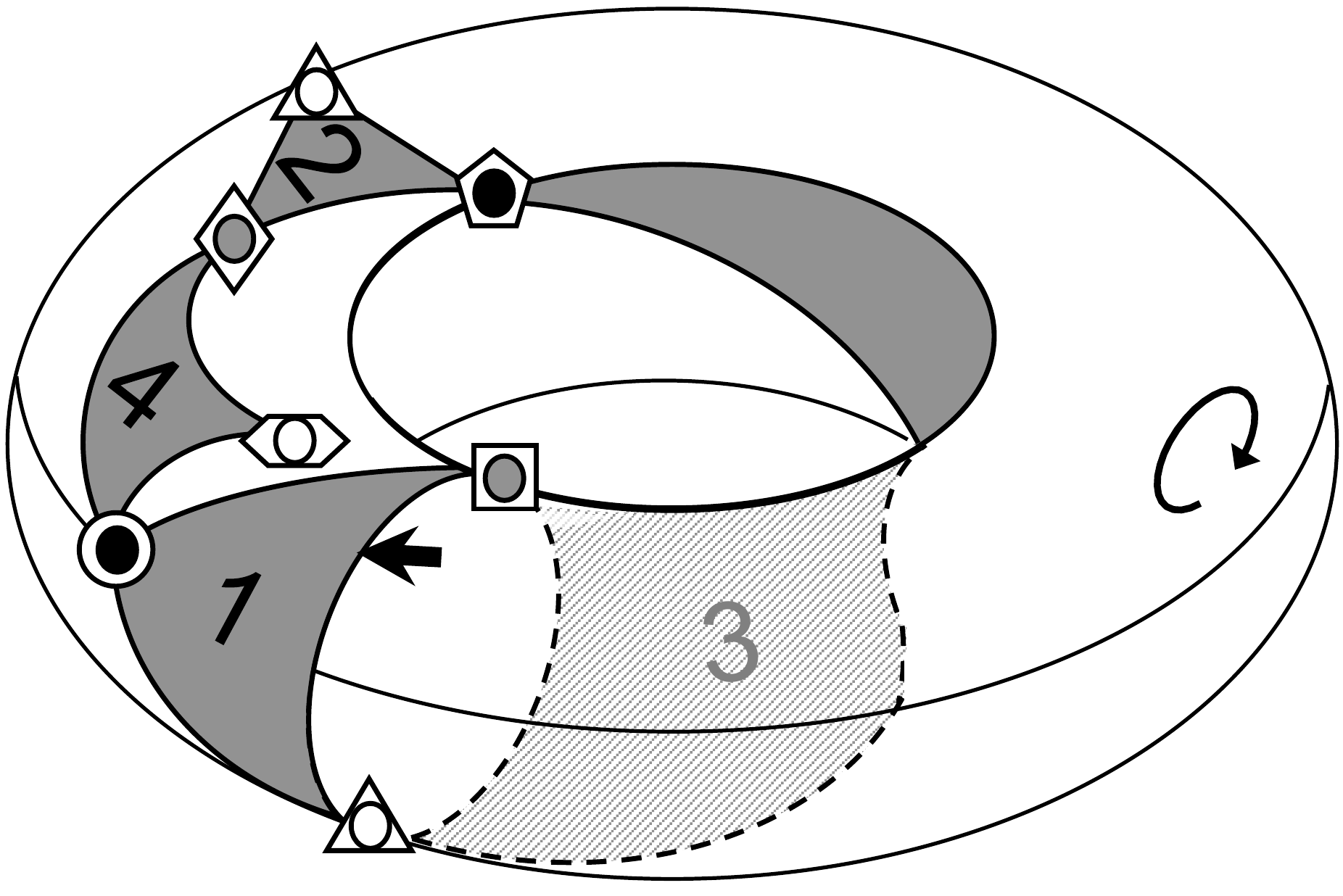}
\caption{The reconstructed partitioned cactus}
\label{bb}
\end{center} 
\end{figure}

\end{example}

\indent Then we show by induction that the relabeling permutations $\lambda_1$, $\lambda_2$, $\lambda_3$ are uniquely determined.
The element $\lambda_1(1)$ is uniquely determined as the minimum element of $\lambda_1(\pi_1^{p_1})$.
Assume that $\lambda_1(l)$, $\lambda_2(l-1)$ and $\lambda_3(l-1)$ have been uniquely determined for $l=1\ldots i$ for a given $i< N$. As $\lambda_3$ is an increasing function on the blocks of $\pi_3$, the element $\lambda_3(i)$ is necessarily the minimal (not yet attributed within the procedure) element of $\lambda_3(\pi_3^{k_i})$, where $k_i$ is the index of the grey block such that $\lambda_1(i) \in \lambda_1(\pi_3^{k_i})$. Hence $\lambda_3(i)$ is uniquely determined. We proceed by computing  $\lambda_1(\alpha_3(i)) = \overline{\sigma}_1^{-1}(\lambda_3(i))$. Then we identify $\lambda_2(i)$ as being the minimal (not yet attributed within the procedure) element of $\lambda_2(\alpha_3^{-1}(\pi_2^{j_i}))$ where $j_i$ is the index of the black block such that $\lambda_1(\alpha_3(i)) \in \lambda_1(\pi_2^{j_i})$. We compute $\lambda_1(\alpha_2\alpha_3(i)) = \overline{\sigma}_2^{-1}(\lambda_2(i))$. Then we identify $\lambda_1(i+1)$ as being the minimal (not yet attributed within the procedure) element of $\lambda_1((\pi_1^{l_i}))$ where $l_i$ is the index of the white block such that $\lambda_1(\alpha_2\alpha_3(i)) \in \lambda_1(\pi_1^{l_i})$. (For $\pi_1$ is stable by $\alpha_1$, the element $i+1= \alpha_1\alpha_2\alpha_3(i)$ belongs to the same white block as $\alpha_2\alpha_3(i)$ and $\lambda_1(i+1)$ is uniquely determined.) Finally, $\lambda_2(N)$ and $\lambda_3(N)$ are uniquely determined.

\begin{example}
Proceeding with our example, we find that the smallest integer of $\lambda_1(\pi_1^2)$ is equal $2$ and as a consequence
$\lambda_1(1) = 2$.
As $2$ belongs to $\lambda_1(\pi_3^1)$ and the smallest (not yet attributed) element of $\lambda_3(\pi_3^1)$ is $1$, we have
$\lambda_3(1) = 1$.
Then, $\lambda_1(\alpha_3^1(1)) = \overline{\sigma}_1^{-1}(1) = 4$ and looking at $\lambda_1(\pi_2)$, we get 
$\lambda_2(1) = 3$.
We compute $\lambda_1(\alpha_2\alpha_3(1)) = \overline{\sigma}_2^{-1}(\lambda_2(1)) = 3$. The minimal not yet attributed element of the block of $\lambda_1(\pi_1)$ containing $2$ is equal to $3$. As a result we obtain
$\lambda_1(2) = 3$.
The full reconstruction yields $\lambda_1 = 2341$, $\lambda_3 = 1324$, $\lambda_2=3124$.


\end{example}

\noindent As a direct consequence, the partitioned cactus is uniquely determined since 
\begin{eqnarray}
\nonumber &\label{rec1} \pi_1 = \lambda_1^{-1}(\lambda_1(\pi_1)),\; \pi_2 = \lambda_1^{-1}\circ\overline{\sigma_2}^{-1}(\lambda_2(\alpha_3^{-1}(\pi_2))),\; \pi_3 = \lambda_1^{-1}\circ\overline{\sigma_1}^{-1}(\lambda_3(\pi_3))&\\
\nonumber  &\label{rec2} \alpha_1 = \gamma_N\circ\lambda_2^{-1}\circ \overline{\sigma_2}\circ\lambda_1,\;\;\; \alpha_2 = \lambda_1^{-1}\circ \overline{\sigma_2}^{-1}\circ \lambda_2\circ\lambda_3^{-1}\circ\overline{\sigma_1}\circ\lambda_1&
\end{eqnarray}

\begin{example}
\noindent We have $\pi_1 = \{\{4\};\{1,2,3\}\}$, $\pi_2 = \{\{1,4\};\{2,3\}\}$, $\pi_3 = \{\{1,3\};\{2,$ $4\}\}$ and $\alpha_1 = (1\,3)(2)(4)$, $\alpha_2 = (1\,4)(2\,3)$
\end{example}
\noindent {\bf Surjectivity}\\
The final step of this proof is to show that $\Theta(p_1, p_2, p_3, N)$ is a surjection. Let $(\tau, S_{0}, S_{1}, S_{2}, \chi, \sigma_{1}, \sigma_{2})$ be any 7-uple of $I(p_1, p_2, p_3, N)$. Provided that $S_0 \cap \chi = \emptyset$ we can always define the partitions $\lambda_1(\pi_1)$, $\lambda_2(\alpha_3^{-1}(\pi_2))$ and $\lambda_3(\pi_3)$ and the permutations $\overline{\sigma_1}$ and $\overline{\sigma_2}$ according to the reconstruction procedure we pointed out in the previous subsection. As a direct consequence, we can always define $\lambda_1(\pi_2)$ and $\lambda_1(\pi_3)$ using that $\lambda_1(\pi_2) = \overline{\sigma_2}^{-1}(\lambda_2(\alpha_3^{-1}(\pi_2)))$ and $\lambda_1(\pi_3) = \overline{\sigma_1}^{-1}(\lambda_3(\pi_3))$. 
To prove surjectivity of $\Theta$ we need to show that the iterative reconstruction of $\lambda_1$, $\lambda_2$, $\lambda_3$ defined in the injectivity proof can always be fulfilled and always gives valid objects as an output. The main difficulty lies in the boundary condition. First of all let us remark that we can always define $\lambda_1(1)$ as the minimal element of $\lambda_1(\pi_1^{p_1})$. Suppose now that we were able to reconstruct $\lambda_1(l)$ for $l = 1\ldots h$ with $h < N$, as well as $\lambda_2(l)$ and $\lambda_3(l)$ for $l = 1\ldots h-1$. We would not be able to reconstruct $\lambda_3(h)$ if and only if all the elements of $\lambda_3(\pi_3^{k_h})$ such that $\lambda_1(h) \in \lambda_1(\pi_3^{k_h})$ have already been allocated during the procedure. Clearly it would mean that $ \mid \lambda_1(\pi_3^{k_h}) \mid+1$ different elements would belong to the set $\lambda_1(\pi_3^{k_h})$ that is an obvious contradiction. The impossibility of reconstructing $\lambda_2(h)$ would be equivalent to the fact that all the elements of $\lambda_2(\pi_2^{j_h})$ with $\overline{\sigma_1}^{-1}(\lambda_3(h)) \in \lambda_1(\pi_2^{j_h})$ have already been allocated. Since $\overline{\sigma_1}$ is a bijection, this leads to a similar contradiction. Same argument applies to reconstruction of $\lambda_1(h+1)$ if $\overline{\sigma_2}^{-1}(\lambda_2(h)) \in \lambda_1(\pi_1^{i_{h+1}})$ with $i_{h+1} \neq p_1$. If $\overline{\sigma_2}^{-1}(\lambda_2(h)) \in \lambda_1(\pi_1^{p_1})$ we can only tell that $\mid \lambda_1(\pi_1^{p_1}) \mid$ different integers belongs to $ \lambda_1(\pi_1^{p_1})$ since $\lambda_1(1)$ was not constructed in the same way as the other $\lambda_1(l)$. However according to our construction, if the black vertex in $\tau$ corresponding to a given block $\pi_{2}^{u}$ is the direct descendant of the white vertex associated with $\pi_{1}^{v}$ then $\overline{\sigma_2}^{-1}(\lambda_2(m_{2}^{'u})) \in \lambda_1(\pi_{1}^{v})$.
Hence, if all the elements of $\lambda_1(\pi_{1}^{v})$ have been used for the reconstruction process, the maximum elements (and therefore all the elements) of all the blocks $\lambda_2(\alpha_3^{-1}(\pi_{2}^{u}))$ corresponding to descending vertices of the white vertex associated with $\lambda_1(\pi_{1}^{v})$ have been used during this process. Similarly we can show that for any vertex of any color, if all of the elements of the corresponding block have been used during the reconstruction process, all the elements of the block corresponding to descending vertices have been used for the reconstruction process. Since $\pi_1^{p_1}$ is associated with the root of the cactus tree $\tau$, if all the elements of $ \lambda_1(\pi_1^{p_1})$ have been used for the reconstruction process, it means that all the elements of all the blocks have been used as well. As a conclusion we state that the reconstruction process came to its expected end and $h=N$ which is a contradiction with our hypothesis.\\
For the final step, let us show that the reconstructed objects are valid 5-tuples of $CC(p_1,p_2,p_3,N)$, i.e. that the stability conditions of the partitions by the permutation are respected. If we have $l \in \pi_1^{i}$, then according to our procedure $\lambda_1^{-1}\circ\overline{\sigma_2}^{-1}\circ\lambda_2(l-1) \in \pi_1^{i}$ for $l \geq 2$. Besides, we showed that $\lambda_1^{-1}\circ\overline{\sigma_2}^{-1}\circ\lambda_2(N) \in \pi_1^{p_1}$ with $1 \in \pi_1^{p_1}$. Hence, for all $l$,  elements $l$ and $\lambda_1^{-1}\circ\overline{\sigma_2}^{-1}\circ\lambda_2\gamma_N^{-1}(l)$ belong to the same block of $\pi_1$ and $\pi_1$ is stable by $\alpha_1^{-1}$ or, equivalently, by $\alpha_1$. 
Similarly, if $\lambda_2(l)\in \lambda_2(\alpha_3^{-1}(\pi_2^{j}))$, i.e. $\lambda_1^{-1}\circ\overline{\sigma_2}^{-1}\circ\lambda_2(l)\in \lambda_1^{-1}\circ\overline{\sigma_2}^{-1}(\lambda_2(\alpha_3^{-1}(\pi_2^{j}))) = \pi_2^{j}$, then $\lambda_1^{-1}\circ\overline{\sigma_1}^{-1}(\lambda_3(l))\in\lambda^{-1} \circ\overline{\sigma_2}^{-1}\circ\lambda_2(\alpha_3^{-1}(\pi_2^{j})) = \pi_2^{j}$. As a result, $\pi_2$ is stable by $\alpha_2 = \lambda_1^{-1}\circ \overline{\sigma_2}^{-1}\circ \lambda_2\circ\lambda_3^{-1}\circ\overline{\sigma_1}\circ\lambda_1$. Trivially, $\pi_3$ is stable by $\alpha_3= \alpha_2^{-1}\circ \alpha_1^{-1}\circ \gamma_N$ since $\alpha_2^{-1}\circ \alpha_1^{-1}\circ \gamma_N(\pi_3) =  \lambda_1^{-1}\circ\overline{\sigma_1}^{-1}(\lambda_3(\pi_3)) = \pi_3 $ according to our definition.
\section{Appendix : cardinality of the sets of cactus trees}
\label{proof}
Enumeration of cactus trees can be easily performed as they are recursive objects and can be decomposed in (i) the white root, (ii) the cactus trees rooted in a black vertex descending from the root, (iii) a triple composed of a black rooted cactus tree, a grey rooted cactus tree, a triangle for each triangle descending  from the root, (iv) the positions of the triangles in the descendant list.
As an example the cactus tree in figure \ref{fig : ex ct} is decomposed on figure \ref{fig : decompLagrange}.
 \begin{figure}[h]
\begin{center}
\includegraphics[width=110mm]{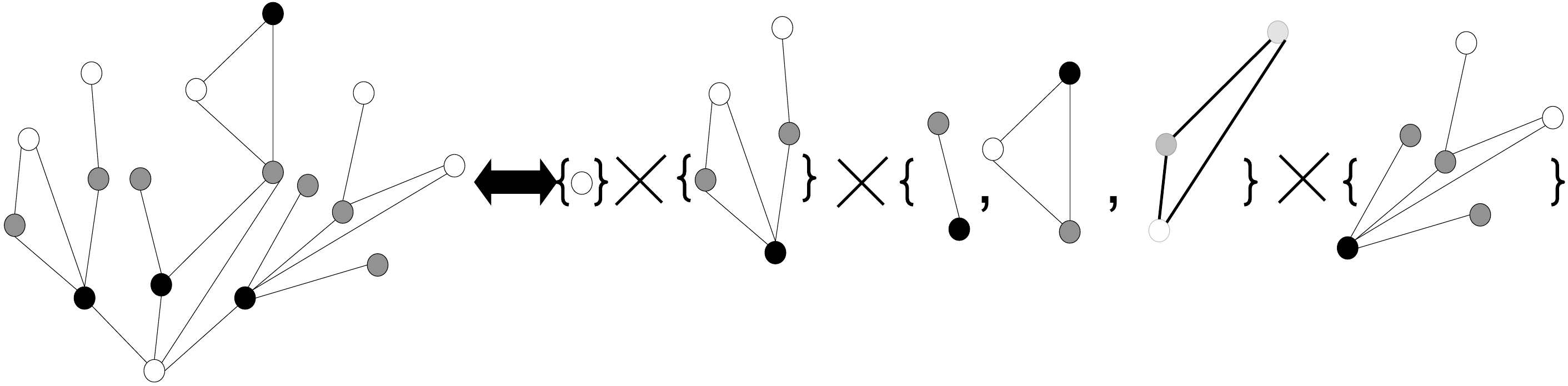}
\caption{Decomposition of the cactus tree in example \ref{ex : ex}}
\label{fig : decompLagrange}
\end{center} 
\end{figure}

\noindent Let $W$ be the generating function of the white rooted cactus trees. We have:
\begin{equation}
W = \sum_{p_1,p_2,p_3,a,b,c}CT(p_{1}, p_{2}, p_{3}, a, b, c)x_1^{p_{1}}x_2^{p_{2}}x_3^{p_{3}} y_1^a y_2^b y_3^c
\end{equation}

\noindent Similarly we define, $B$ (resp. $G$), the generating function of the black rooted cactus trees (resp. grey rooted). Finally,
we denote $T_1$ (resp. $T_2$, $T_3$) the generating function of a grey rooted triangle (resp. white, black). Obviously $T_i = y_i$ and using the previous decomposition, we have :
\begin{equation}
\nonumber W = x_1\sum_{k \geq 0}\sum_{0 \leq j \leq k}\sum_{{\bf i}}B^{k-j}(BGT_2)^{j}
\end{equation}
\noindent where {\bf i} is an integer vector indicating the positions of the $j$ triangles among the $k$ descendants from the root. Going ahead gives:
\begin{eqnarray}
\label{eq W}   \nonumber W = x_1\sum_{k \geq 0}B^k\sum_{0 \leq j \leq k}\binomial{k}{j}(GT_2)^j =\frac{ x_1}{1-B(1+GT_2)}= x_1.\Phi_1(W,B,G,T_1,T_2,T_3)
\end{eqnarray}
\noindent Similarly,
\begin{eqnarray}
\nonumber B = \frac{x_2}{1-G(1+WT_3)}  \label{eq B} = x_2.\Phi_2(W,B,G,T_1,T_2,T_3)\\
\nonumber G = \frac{x_3}{1-W(1+BT_1)} \label{eq G} = x_3.\Phi_3(W,B,G,T_1,T_2,T_3)
\end{eqnarray}
As a direct consequence, the generating functions $W,B,G,T_1,T_2,T_3$ are solutions of a functional equation of the type $(W,B,G,T_1,T_2,T_3) = {\bf x}.{\bf \Phi}(W,B,G,T_1,T_2,T_3)$, where ${\bf x} = (x_1,x_2,x_3,y_1,y_2,y_3)$ and ${\bf \Phi} = (\Phi_i)_{1\leq i \leq 6}$ with $(\Phi_i) = 1$ for $4\leq i \leq 6$.\\
\noindent Using the multivariate Lagrange formula for monomials (see \cite{GJ}), we find:
\footnotesize
\begin{eqnarray}
CT(p_{1}, p_{2}, p_{3}, a, b, c) =\sum_{\{\mu_{ij}\}}\frac{\mid\mid \delta_{ij}k_j-\mu_{ij}\mid\mid}{p_1 p_2 p_3 a b c}\prod_{1\leq i \leq 6}[W^{\mu_{i1}}B^{\mu_{i2}}G^{\mu_{i3}}T_1^{\mu_{i4}}T_2^{\mu_{i5}}T_3^{\mu_{i6}}]\Phi_i^{k_i}
\label{eq lagrange}
\end{eqnarray}
\normalsize
\noindent where $(k_1,k_2,k_3,k_4,k_5,k_6) = (p_{1}, p_{2}, p_{3}, a, b, c)$ and the sum is over all $6\times 6$ integer matrices $\{\mu_{ij}\}$ such that 
\begin{itemize}
\item $\mu_{11} = \mu_{14} = \mu_{16} = \mu_{22} = \mu_{24}=\mu_{25} = \mu_{33} =\mu_{35}=\mu_{36} = 0=\mu_{ij}$ for $i\geq 4$
\item $\mu_{21} + \mu_{31} = p_1 - 1$, $\mu_{12} + \mu_{32} = p_2$, $\mu_{13} + \mu_{23} = p_3$, $\mu_{34} = a$, $\mu_{15} = b$, $\mu_{26} = c$
\end{itemize}
Looking for zero contribution terms in expression (\ref{eq lagrange}), we notice that $G$ and $T_2$ have necessarily the same degree in the formal power series expansion of $\Phi_1$ (see equation (\ref{eq W})). Hence, a non zero contribution of $[W^{\mu_{11}}B^{\mu_{12}}G^{\mu_{13}}T_1^{\mu_{14}}T_2^{\mu_{15}}T_3^{\mu_{16}}]\Phi_1^{k_1}$ implies $\mu_{13} = \mu_{15} = b$. Similar remarks gives non zero contributions only for $\mu_{21} = c$ and $\mu_{32} = a$. As a result only one matrix $\mu$ yields a non zero contribution.\\
\noindent For this particular $\mu$, $\mid\mid \delta_{ij}k_j-\mu_{ij}\mid\mid = abc(p_2 p_3 - (p_3-b)a)$. Then notice that :
\begin{eqnarray}
\nonumber \Phi_1^{p_1} = \sum_{u} \binomial{p_1-1+u}{u}(B(1+GT_2))^{u}= \sum_{u,v} \binomial{p_1-1+u}{u}B^u\binomial{u}{v}(GT_2)^{v}
\end{eqnarray}
As a result, the coefficient in $W^{\mu_{11}}B^{\mu_{12}}G^{\mu_{13}}T_1^{\mu_{14}}T_2^{\mu_{15}}T_3^{\mu_{16}}$ is equal to \linebreak $\binomial{p_1+p_2-1-a}{p_1-1,\,\,p_2-a-b}$.
Using the same argument we have :
\begin{eqnarray*}
[W^{\mu_{21}}B^{\mu_{22}}G^{\mu_{23}}T_1^{\mu_{24}}T_2^{\mu_{25}}T_3^{\mu_{26}}]\Phi_2^{p_2} &=& \binomial{p_2+p_3-1-b}{p_2-1,\,\,p_3-b-c}\\
\left [W^{\mu_{31}}B^{\mu_{32}}G^{\mu_{33}}T_1^{\mu_{34}}T_2^{\mu_{35}}T_3^{\mu_{36}}\right ] \Phi_3^{p_3} &=& \binomial{p_1+p_3-2-c}{p_3-1,\,\,p_1-1-a-c}
\end{eqnarray*}
Putting everything together yields the desired result.

\bibliographystyle{amsalpha}

\end{document}